\documentclass[11pt]{article}

\usepackage{amsmath, amsthm, amssymb, amsfonts, mathtools}
\pdfoutput=1
\usepackage[margin=1in]{geometry}
\usepackage{enumitem}
\usepackage{mathrsfs}
\usepackage{graphicx}
\usepackage{xcolor}
\usepackage{hyperref}

\usepackage{float}

\allowdisplaybreaks

\usepackage{titlesec}
\titlespacing*{\section}{0pt}{1.0ex plus 0.5ex minus 0.2ex}{0.8ex}
\titlespacing*{\subsection}{0pt}{0.8ex plus 0.3ex minus 0.2ex}{0.6ex}
\titlespacing*{\subsubsection}{0pt}{0.6ex plus 0.2ex minus 0.1ex}{0.4ex}

\setlist{nosep}

\newtheorem{theorem}{Theorem}[section]
\newtheorem{lemma}[theorem]{Lemma}
\newtheorem{proposition}[theorem]{Proposition}

\newtheorem*{acknowledgements}{Acknowledgements}

\theoremstyle{definition}
\newtheorem{definition}[theorem]{Definition}
\newtheorem{remark}[theorem]{Remark}

\numberwithin{equation}{section}

\makeatletter
\g@addto@macro\th@plain{\thm@preskip=4pt \thm@postskip=4pt}
\g@addto@macro\th@definition{\thm@preskip=4pt \thm@postskip=4pt}
\makeatother

\date{}

\begin{document}
\title{Solutions to Mean Curvature Flow with Uniform Bounds on the Mean Curvature and Its Gradient}
\author{Priyamvada Vishwamitra}

\maketitle

\begin{abstract}
\noindent
In the setting of a complete, smooth properly immersed mean curvature flow, we assume uniformly bounded $|H|$ and $|\nabla H|$ on $M^n\times[0,T)$ and some bounded initial geometry to get local spatial $L^p$ estimates for the second fundamental form with $p\in[4,\infty)$. For $p>n+2$, this leads us to a local space time $L^\infty$ bound for the second fundamental form which allows us to smoothly extend the flow $F:M^n\times [0,T) \rightarrow \mathbb{R}^{n+1}$ past the singular time $T<+\infty$ for a short time. 
\end{abstract}

\section{Introduction}
Let $M$ be a closed, smooth, $n$-dimensional manifold and $F_0:M\rightarrow\mathbb{R}^{n+1}$ be a smooth immersion of $M$ into $\mathbb{R}^{n+1}$. We call a one-parameter family of smooth immersions $F_t(\cdot):=F(\cdot,t):M\rightarrow\mathbb{R}^{n+1}, t\in[0,T)$ a mean curvature flow if it satisfies
\begin{equation}
    \begin{cases}
        \frac{\partial}{\partial t} F(x,t) &= \Vec{H}(x,t),\\
        F(x,0)&=F_0(x)
    \end{cases}
\end{equation}
$\forall x\in M,t\in[0,T)$ where $\Vec{H}(x,t)=-H(x,t)\nu(x,t)$ is the mean curvature vector of the immersed hypersurface $M_t=F_t(M)$ at the point $F(x,t)$ such that $H(x,t)$ is the scalar mean curvature and $\nu(x,t)$ is the chosen unit normal along $M_t$.  

\cite{huisken1984flow} showed that the second fundamental form $A$ blows up at the singular time $T<+\infty$.
\begin{equation*}
    \limsup_{t\rightarrow T}\sup_{x\in M} |A|(x,t) = \infty
\end{equation*}
Then the question is under what curvature conditions can we exclude singularities during the flow? The extension problem mentioned in 2.4.10 of \cite{book} asks whether a uniformly bounded mean curvature on $M^n\times[0,T)$ is sufficient to extend the mean curvature flow (1.1) smoothly beyond time $T$? In most instances, progress toward resolving this problem has relied on imposing additional hypotheses slightly stronger than the mean curvature bound.

\cite{Le_Sesum_2009} showed extension results via a soft reverse Hölder inequality and Moser iteration techniques by assuming an optimal space time integral bound on $H$ and a lower bound $h_{ij}\geq -Bg_{ij}, B\geq 0, B\in\mathbb{R}$. This was key is bounding $|H|$ and arriving at a contradiction after a blow-up argument. To this end, of further interest is \cite{Le_Sesum_2010}, \cite{Le_Sesum_2011}, \cite{Lin_Sesum_2016} in which several extension and regularity theorems are given under various different assumptions for surfaces, solutions developing Type I singularities in arbitrary dimension. For $n=2$, \cite{Li_Wang_2019} confirmed the blow up of $H$ without any further conditions. For $3\leq n\leq 6$ recently, \cite{han2025singularitymeancurvatureflow} investigated singularities of weak MCFs with bounded $H$, bounded Morse index and showed that extension is possible; more precisely that either $H$ or the Morse index blows up at the first singular time. For $n\geq 7$, in general, \cite{stolarski2023existence} proved otherwise i.e that $H$ need not necessarily blow up. A non trivial subset of Velázques's solutions to the MCF is the counter example, they develop finite time singularities with uniformly bounded $H$; furthermore the existence of compact solutions satisfying such geometry was shown by the author. In the paper of \cite{Wang_2022} a uniform bound on $|HA|$ was assumed so as to obtain an extension theorem for (1.1). There the quantity $HA$ was interesting because it describes the evolution of the metric in mean curvature flow. It was also pointed out to use the following evolution equation
\begin{equation}
    \partial_t |A|^2 = 2(\nabla^2 H \cdot A + H\cdot \text{tr}(A^3)); 
\end{equation} 
in Lemma 2.2 we give a derivation of (1.2). Using this, a differential inequality for $\int_M |A|^p$ was computed; after integration and applying Moser iteration in the regime of $p>n+2,\: \: \sup |A|$ was locally bounded in terms of the initial geometry and the uniform $|HA|$ bound. 
The original inspiration comes from \cite{KOTSCHWAR20162604} which proved extension results for the Ricci flow: a bound on the Ricci curvature implies a polynomial growth of the Riemannian curvature due to the evolution equations $\partial_t |Rm|^2 = \nabla^2 Ric\ast Rm + Ric \ast Rm \ast Rm, \partial_t g = -2Ric(g)$.
Here, in this paper, we work with information related to the mean curvature alone as our hypotheses. In section 3, we assume uniform space time bounds on $|H|$, $|\nabla H|$ and calculate spatial $L^p$ bounds for $|A|$ - similar to \cite{Wang_2022} - in the case of compact, smooth, immersed MCFs. In section 4, we obtain a localized version of the estimate for complete, smooth, properly immersed MCFs. In section 5, the local space time supremum bound on $|A|$ is computed following Moser iteration techniques from \cite{Le_Sesum_2009}. We make the following abuse of notation to understand what integration on the hypersurface means: Consider $M_t\cap B_r(x_0)$ 
for any $t\in [0,T)$ around a point $x_0\in\mathbb{R}^{n+1}$ and radius $r>0$, then  
\begin{equation*}
    \int_{M_t\cap B_r(x_0)}|A|^2(t) d\mu_t = \int_{F_t^{-1}(B_r(x_0))}|A|^2(t) d\mu_t =: \lVert A(t) \rVert_{2,F_t^{-1}(B_r(x_0))}^2
\end{equation*}
 see Definition 2.3 for the exact definition of the norms. On the domain manifold, we can take a pull-back of this patch and express 
\begin{equation*}
    M\supset\Omega_t:=F_t^{-1}(B_r(x_0))
\end{equation*} 
as some open set. We define this pre-image set in Definition 2.5. Despite the fact that this pre-image can have multiple connected components, we have the following set-level inclusion that immediately gives a control on the volume of such sets because they satisfy the evolution equation of the measure under the mean curvature flow
\begin{equation*}
    \Omega_t \subseteq \hat{\Omega}_0 \Rightarrow \text{vol}_{g_t}(\Omega_t) \leq \text{vol}_{g_0}(\hat{\Omega}_0).
\end{equation*}
Here $\hat{\Omega}_0=F_0^{-1}(B_{r+CT}(x_0))$ is the pre-image at the initial time of a slightly larger ball. This follows from $|H|(x,t)\leq C,\forall x\in M, t\in [0,T)$ which gives a quantitative bound for the speed of the evolution; this means a point cannot move more than linearly at a time. The details of this is given in Proposition 2.6. Our main theorems are, 
\begin{theorem}
Let $F:M^n \times [0,T)\rightarrow\mathbb{R}^{n+1}$ be a complete, smooth, properly immersed mean curvature flow. Assume the following uniform bounds \[ \sup_{M^n \times [0,T)} |H|(x,t) \leq c_1 < \infty , \quad \quad \quad \sup_{M^n \times [0,T)} |\nabla H|(x,t) \leq c_2 < \infty .\] Let $p\in[4,\infty), x_0\in\mathbb{R}^{n+1}, 0<r\leq 1, K:=Tc_1, \hat{\Omega}_0:=F_0^{-1}(B_{r+K}(x_0))$. Define \[ A_0:=\lVert A(0) \rVert_{p,\hat{\Omega}_0}, \quad \quad V_0:=\text{vol}_{g_0}(\hat{\Omega}_0)=\int_{F_0^{-1}(B_{r+K}(x_0))} d\mu_0.\] If $A_0<\infty, V_0<\infty$, then there exists constants $\tilde{k}=\tilde{k}(n,p,c_1,c_2)$, $\tilde{C}_1=\tilde{C}_1(p,c_1,c_2,r)$ such that for every $t\in(0,T)$,
    \[\int_{M_t\cap B_{r/2}(x_0)} |A|^p(t) d\mu_t \equiv \int_{F_t^{-1}(B_{r/2}(x_0))}|A|^p(t) d\mu_t \leq 2e^{\tilde{k}t}\biggl(\int_{F_0^{-1}(B_{r+K}(x_0))} |A|^p(0) d\mu_0 + \tilde{C}_1V_0\biggr).\]
\end{theorem}
\begin{theorem}
Let $F:M^n\times[0,T)\rightarrow\mathbb{R}^{n+1}$ be a complete, smooth, properly immersed mean curvature flow. Assume that the uniform bounds on $|H|, |\nabla H|$ from theorem 1.1 hold. Let $s>n+2, T\geq 1, K:= Tc_1, x_0\in\mathbb{R}^{n+1}, \hat{\Omega}_0 := F^{-1}_0(B_{2+K}(x_0))$. Define \[ A_0:= \lVert A(0) \rVert_{s,\hat{\Omega}_0}, \quad \quad V_0:=\text{vol}_{g_0}(\hat{\Omega}_0)=\int_{F_0^{-1}(B_{2+K}(x_0))} d\mu_0.\] If $A_0<\infty, V_0<\infty$, then there exists constants $\tilde{E}=\tilde{E}(n,s,T,c_1,V_0), \tilde{k}=\tilde{k}(n,s,c_1,c_2)$ and $\tilde{C}_1=\tilde{C}_1(s,c_1,c_2)$ such that \[ \sup_{t\in[\frac{1}{2},1]}\sup_{F_t^{-1}(B_{1/2}(x_0))} |A| \leq \tilde{E}\biggl\{1+\left( \int_{F_0^{-1}(B_{2+K}(x_0))}|A|^s(0) d\mu_0 + \tilde{C}_1V_0\right)^{\tilde{s}}\left(\int_0^1 e^{\tilde{k}t} dt\right)^{\tilde{s}}\biggr\} \]
     where $\tilde{s}=\tilde{s}(n,s)$.
\end{theorem}
\begin{theorem}
   Let $F:M^n\times[0,T)\rightarrow\mathbb{R}^{n+1}$ be a complete, smooth, properly immersed mean curvature flow. Assume that the uniform bounds on $|H|, |\nabla H|$ from theorem 1.1 hold. Let $s>n+2, T> 1, \rho>2\sqrt{2}T$ be given. For all $x_0\in\mathbb{R}^{n+1}$ define $\hat{\Omega}_0(x_0):= F_0^{-1}(B_{\rho}(x_0))$, \[ A_0(x_0):=\lVert A(0) \rVert_{s,\hat{\Omega}_0(x_0)}, \quad \quad V_0(x_0):=\text{vol}_{g_0}(\hat{\Omega}_0(x_0))=\int_{F_0^{-1}(B_{\rho}(x_0))}d\mu_0.\] If $\mathcal{A}:=\sup_{x_0\in\mathbb{R}^{n+1}}A_0(x_0)<\infty$,$\mathcal{V}:=\sup_{x_0\in\mathbb{R}^{n+1}}V_0(x_0)<\infty$, then $\exists \kappa=\kappa(n,s,T,c_1,c_2,\rho,\mathcal{A},\mathcal{V})$ and a $T_0=T_0(n,\rho,\kappa)>0$ such that the flow admits a  complete, smooth, properly immersed extension on the interval $[T,T+T_0)$, which is unique among the class of such extensions with $\sup_{M^n\times[T/2,T+T_0)}|A|<\infty$.
\end{theorem}
In this paper these theorems appear as Theorem 4.1, Theorem 5.4 and Theorem 5.6 respectively.
\begin{acknowledgements}
    I thank Prof. Dr. Miles Simon for many helpful, insightful discussions in this direction along with thorough feedbacks.
\end{acknowledgements}
\section{Preliminaries}
Let $F:M^n\times [0,T) \rightarrow \mathbb{R}^{n+1}$ be a closed, smooth, immersed mean curvature flow. We define the time dependent induced metric, volume form and second fundamental via the pull back of $F_t$,
\begin{align*}
    g(\cdot, t) &\equiv g_t(\cdot) := F_t^*\delta_{\text{Eucl}}\\
    A(\cdot, t) &\equiv A_t(\cdot) := (\Bar{\nabla} dF_t)^\perp
\end{align*}
which we write locally as 
\begin{align*}
    g_{ij}(x,t) &=\langle \partial_iF_t(x),\partial_j F_t(x)\rangle\\
    A_{ij}(x,t) &= \langle \Bar{\nabla}_{{\partial_i}F_t(x)}\nu(x),\partial_jF_t(x)\rangle \\
    d\mu_t \equiv &\sqrt{g_t} dx= \sqrt{\det g_{ij}(x,t)}dx
\end{align*}
where $\Bar{\nabla}$ is the standard covariant derivative from the ambient $\mathbb{R}^{n+1}$. We now recall the well known evolution equations:
\begin{lemma}
    (Corollaries 3.5,3.6, Huisken,\cite{huisken1984flow}). For a closed, smooth immersed mean curvature flow we have:
    \begin{align}
        \partial_t d\mu_t &= -H^2 d\mu_t\\
    \partial_t H^2 = \Delta H^2 &- 2|\nabla H|^2 + 2|A|^2H^2\\
        \partial_t |A|^2 = \Delta &|A|^2 - 2|\nabla A|^2 + 2|A|^4
    \end{align}
\end{lemma}
\noindent We frequently use Kato's inequality in our computations in sections 3 and 4: Since $A$ is a smooth tensor, it satisfies the pointwise bound $|\nabla|A||\leq |\nabla A|$ whenever $|A|\neq 0$.\\
\\
\noindent The MCF is invariant under the parabolic rescaling, hence we set
\[ \Tilde{F}(\Tilde{x},\Tilde{t}) = \lambda(F(x,\lambda^{-2}\tilde{t}+t_0) - x_0)\]
$\forall (x,\Tilde{t})\in M\times[-\lambda^2t_0,\lambda^2(T-t_0)]$ and some $x_0\in\mathbb{R}^{n+1}$, $t_0\in[0,T)$ and $\lambda>0$. 
\[ \Tilde{A}=\lambda A, \quad \quad \Tilde{g}=\lambda^2 g, \quad \quad |\Tilde{A}| = \lambda^{-1}|A|, \quad \quad \Tilde{H}= \lambda^{-1}H, \quad \quad |\Tilde{\nabla}\Tilde{H}| =\lambda^{-2}|\nabla H|\]
where $\Tilde{A},\Tilde{g},\Tilde{H}$ denote the second fundamental form, induced metric and mean curvature of the rescaled hypersurface $\Tilde{F}_{\tilde{t}}(M)$.
\begin{lemma}
    \begin{equation}
        \partial_t |A|^2 = 2(\nabla^2 H\cdot A + H\cdot\text{tr}(A^3))
    \end{equation} where $\nabla^2H\cdot A=g^{ik}g^{jl}\nabla_i\nabla_jHh_{kl}, \text{tr}(A^3)=h^j_ih^k_jh^i_k$.
\end{lemma}
\begin{proof}
    From Theorem 3.4 of \cite{huisken1984flow}, in local coordinates we have \[ \partial_th_{ij} = \nabla_i\nabla_j H - Hh_{ir}g^{rs}h_{sj} \] 
    Recalling that $\partial_tg_{ij}=-2Hh_{ij}$ we get the following
    \begin{align*}
        \partial_t|A|^2&=\partial_t(g^{ik}g^{jl}h_{ij}h_{kl})\\
        &=2g^{ik}g^{jl}(\partial_th_{ij})h_{kl} + 2(\partial_tg^{ik})g^{jl}h_{ij}h_{kl}\\
        &= 2g^{ik}g^{jl}(\nabla_i\nabla_jH -Hh_{ir}g^{rs}h_{sj})h_{kl} + 2(2Hh^{ik})g^{jl}h_{ij}h_{kl}\\
        &= 2g^{ik}g^{jl}\nabla_i\nabla_jHh_{kl} - 2Hg^{ik}g^{jl}g^{rs}h_{ir}h_{sj}h_{kl} + 4 Hg^{ai}g^{bk}g^{jl}h_{ab}h_{ij}h_{kl}
    \end{align*}
    after relabeling the indices, the contractions in the last two terms equal $h^j_ih^k_jh^i_k$, we can choose an orthonormal basis such that $g_{ij}=\delta_{ij}, h_{ij}=\kappa_ig_{ij}$ where $\kappa_i$ are the principal curvatures; then\[ \partial_t |A|^2 = 2\sum_{i=1}^n\nabla_i\nabla_iH\cdot\kappa_i + 2 H\sum_{i=1}^n\kappa_i^3 \] and (2.4) holds. We freely utilize the estimate $|\text{tr}(A^3)| \leq c_n|A|^3$ without further mention. The advantage of this evolution equation lies in its control of nonlinearities via $|A|^3$ terms, rather than the more singular $|A|^4$ term present in (2.3). This reduction in order makes it feasible to derive apriori $L^p$ bounds on $A$ in sections 3 and 4.
\end{proof}
\begin{definition}
    Let $F:M^n\times [0,T) \rightarrow \mathbb{R}^{n+1}$ be a closed, smooth immersed mean curvature flow. The total volume of $M_t=F_t(M)$ with multiplicity is the integral over $M$ of the pulled back measure $d\mu_t$ induced by $F_t$ and is given by,
    \begin{align}
        \text{vol}_{g_t}(M_t) := \int_M d\mu_t = \int_M \sqrt{g_t}dx < +\infty.
    \end{align}
    Let $h\in C_C^{\infty}(M\times [0,T))$, then 
    \begin{align}
       \lVert h \rVert_{1,M_t} := \int_{M_t} |h|d\mu_t = \int_M |h|(x,t)d\mu_t.
    \end{align}
\end{definition}
\begin{definition}
    The $L^p$ space time integral for $h\in C_C^{\infty}(M\times [0,T))$ over the evolving hypersurface $t\mapsto M_t=F_t(M),t\in[0,T)$ is then defined as 
    \begin{align*}
        \lVert h \rVert_{p,M_{(\cdot)},[0,T)} :=  \biggl(\int_0^T\int_M |h|^p(x,t) d\mu_t dt\biggr)^{1/p}
    \end{align*}
\end{definition}
\begin{definition}
    (Pre-image set)  Let $F:M^n\times [0,T) \rightarrow \mathbb{R}^{n+1}$ be a complete, smooth, properly immersed mean curvature flow, $M_t=F_t(M)$; for a fixed $x_0\in\mathbb{R}^{n+1},r>0$, the pre-image set of $M_t\cap B_r(x_0)\subset \mathbb{R}^{n+1}$ under $F$ at time $t\in(0,T)$ is defined as
    \begin{equation*}
        \Omega_t := F_t^{-1}(B_r(x_0)) = \{ x \in M : F_t(x) \in M_t\cap B_r(x_0) \} \subset M
    \end{equation*}
    which means the following integrals are well-defined
    \begin{equation*}
         \int_{M_t\cap B_r(x_0)}|A|^2(t) d\mu_t := \int_{F_t^{-1}(B_r(x_0))}|A|^2(t) d\mu_t \equiv \int_{\Omega_t} |A|^2(t) d\mu_t.
    \end{equation*}
\end{definition}
In the next proposition, we convey the idea that every point $x\in M$ that maps into $B_r(x_0)$ at time $t$ must have been mapped into a slightly larger ball around $x_0$ at the initial time.
\begin{proposition}
    Let $F:M^n\times [0,T) \rightarrow \mathbb{R}^{n+1}$ be a complete, smooth, properly immersed mean curvature flow, $M_t=F_t(M)$ and \[ \sup_{M^n\times [0,T)} |H|(x,t) \leq c_1 < +\infty \] 
    then for a fixed $x_0\in\mathbb{R}^{n+1},r>0$ at time $t\in (0,T)$, \[ \Omega_t \subseteq \hat{\Omega}_0 := F_0^{-1}(B_{r + K}(x_0)) \]
    where $K:=\int_0^T |H|(x,t)dt \leq Tc_1$
    \begin{proof}
        Since the flow evolves by mean curvature, 
        \begin{align*}
            \left|\frac{\partial F}{\partial t}(x,t)\right| &= |H|(x,t) \leq c_1\\
            \Rightarrow |F(x,0) - F(x,t)| &\leq \int_0^t |H|(x,s)ds \leq tc_1
        \end{align*}
        for every $0<t<T$. We now have $F_0(x) \in B_{tc_1}(F_t(x))$.
    Let $x\in \Omega_t$, then $F_t(x) \in B_r(x_0)$,
    \begin{equation*}
        |F(x,0) - x_0| \leq |F(x,0)- F(x,t)| + |F(x,t) - x_0| \leq r + tc_1 < r + Tc_1 = r + K.
    \end{equation*}
    Thus $F_0(x)\in B_{r+K}(x_0) \Rightarrow x \in \hat{\Omega}_0$ and the inclusion holds.
    \end{proof}
\end{proposition}
\begin{definition}
    Let $F:M^n \times [0,T) \rightarrow \mathbb{R}^{n+1}$ be a complete, smooth, properly immersed mean curvature flow, $M_t=F_t(M)$; for $x_0\in\mathbb{R}^{n+1},r>0$ and $\Omega_t$ as defined in Definition 2.5,
    \begin{align}
        vol_{g_t}(M_t\cap B_r(x_0)) = \int_{M_t\cap B_r(x_0)} d\mu_t := \int_{F_t^{-1}(B_r(x_0))} d\mu_t \equiv \int_{\Omega_t}d\mu_t = \text{vol}_{g_t}(\Omega_t).
    \end{align}
    To make computations locally on $M_t\cap B_r(x_0)$, we define an intrinsic smooth cutoff function as follows: Choose a function $\varphi\in C_C^{\infty}([0,\infty))$ such that,
    \begin{align*}
        \varphi(s)&=\begin{cases}
            1, \quad \quad \quad s\leq \frac{r^2}{4}\\
            \in[0,1], \:\:\: \frac{r^2}{4}\leq s\leq r^2\\
            0, \quad \quad \quad s\geq r^2
        \end{cases}
    \end{align*}
    and $|\varphi'|\leq c_n/r^2,|\varphi''|\leq c_n/r^4$. For $0<\delta<1$ define
    \begin{align*}
         \psi&:=\varphi^{1/\delta},\\
         \phi_t(x) &:= \psi(|F(x,t)-x_0|^2)        
    \end{align*}
    where $\phi_t:M\rightarrow[0,1], t\in[0,T)$ satisfies
    \begin{align*}
    \phi_t(x) &= \begin{cases}
        1, \quad \quad \quad \: \: |F(x,t) - x_0|^2 \leq \frac{r^2}{4}\\
        \in (0,1), \quad \frac{r^2}{4}\leq |F(x,t)-x_0|^2\leq r^2\\
        0, \quad \quad \quad  \: \: |F(x,t)-x_0|^2 \geq r^2.
    \end{cases}
    \end{align*}
Furthermore, for all $t\in[0,T)$ it is true that $\text{supp}\phi_t\subseteq \Omega_t\subset M$. 
\end{definition}
Proposition 2.6 and Definition 2.7 imply that $\Omega_t$ has only finitely many connected components contained in $\hat{\Omega}_0$. 
\begin{lemma}
For $\Omega_t$ and $\hat{\Omega}_0$ as defined Definition 2.5 and Proposition 2.6 respectively, 
    \begin{equation*}
        \text{vol}_{g_t}(\Omega_t) \leq \text{vol}_{g_0}(\hat{\Omega}_0).
    \end{equation*}
    \begin{proof}
        For every $t\in(0,T)$ Proposition 2.6 gives us $\Omega_t\subseteq \hat{\Omega}_0 \Rightarrow \text{vol}_{g_t}(\Omega_t) \leq \text{vol}_{g_t}(\hat{\Omega}_0)$ and using (2.1) we have $\text{vol}_{g_t}(\hat{\Omega}_0) \leq \text{vol}_{g_0}(\hat{\Omega}_0)$. Hence the result holds.
    \end{proof}
\end{lemma}
\begin{remark}
    For a smooth cutoff function $\phi_t\in C_C^\infty(M)$ defined such that $\text{supp}\phi_t\subseteq\Omega_t\subseteq\hat{\Omega}_0$ with $\phi_t(\cdot)\leq\mathbf{1}_{\Omega_t}(\cdot)$, we have for every $t\in[0,T)$, \begin{equation}\int_M \phi_td\mu_t \leq \text{vol}_{g_t}(\Omega_t) \leq \text{vol}_{g_0}(\hat{\Omega}_0).\end{equation}
\end{remark} 
\begin{lemma}
    For a smooth cutoff function $\phi_t\in C_C^{\infty}(M)$ as in Definition 2.7, we have the following bounds 
    \begin{align}
        |\nabla \phi_t| &\leq \frac{\tilde{C}(n,\delta)}{r}\phi_t^{1-\delta},\\
        |\partial_t \phi_t| &\leq \frac{\tilde{C}(n,\delta)}{r}|H|\phi_t^{1-\delta}.
    \end{align}
    \begin{proof}
    In the compactly supported region i.e $|F(x,t)-x_0|^2 \leq r^2$,
        \begin{align*}
            |\nabla \phi_t (x)|_{g_t} &\leq \frac{1}{\delta} \phi_t^{1-\delta} |\varphi'| |\nabla|F(x,t)-x_0|^2_{\mathbb{R}^{n+1}}|_{g_t}\\
        &\leq \frac{2c_n}{\delta r^2} \phi_t^{1-\delta} |\langle F(x,t)-x_0, \nabla F(x,t)\rangle|_{g_t}\\
        &\leq \frac{2c_n}{\delta r^2} \phi_t^{1-\delta} |F(x,t)-x_0|_{\mathbb{R}^{n+1}}|\nabla F(x,t)|_{g_t}.
        \end{align*}
    Recall that $\nabla F^k\in T_xM$ for each $1\leq k\leq n+1$. Choosing normal coordinates at the point $x\in M$ for any $t\in[0,T)$ we calculate,
    \begin{align*}
    |\nabla F(x,t)|^2_{g_t} &= g^{ij}\sum_{k=1}^{n+1} \partial_iF^k(x,t)\partial_jF^k(x,t)\\
    &= \sum_{i=1}^n\sum_{k=1}^{n+1} \partial_iF^k(x,t)\partial_iF^k(x,t)\\
    &= \sum_{i=1}^n g_{ii}(x,t)\\
    &= \sum_{i=1}^n \delta_{ii} = n.
    \end{align*}
   It follows that,
    \begin{equation*}
        \Rightarrow |\nabla \phi_t(x)|_{g_t} \leq \frac{\tilde{C}(n,\delta)}{r}\phi_t^{1-\delta}
    \end{equation*}
    Similarly,
    \begin{align*}
        |\partial_t\phi_t(x)|_{g_t} &\leq \frac{1}{\delta}\phi_t^{1-\delta}|\varphi'||\partial_t|F(x,t)-x_0|^2_{\mathbb{R}^{n+1}}|_{g_t}\\
        &\leq \frac{2c_n}{\delta r^2}\phi_t^{1-\delta}|\langle F(x,t) - x_0, \partial_t F(x,t)\rangle|_{g_t}\\
        &\leq \frac{2c_n}{\delta r}|H|\phi_t^{1-\delta} = \frac{\tilde{C}(n,\delta)}{r}|H|\phi_t^{1-\delta}
    \end{align*}
    in particular if $|\partial_t F(x,t)| = |H(x,t)| \leq c_1 \leq \frac{1}{\sqrt{2}}$ 
    \begin{align}
        |\partial_t\phi_t(x)|_{g_t} \leq \frac{\tilde{C}(n,\delta)}{r}\phi_t^{1-\delta}
    \end{align}
    \end{proof}
\end{lemma}
\begin{lemma}
    Let $F:M^n \times [0,T)\rightarrow \mathbb{R}^{n+1}$ be a complete, smoooth, properly immersed mean curvature flow. For $x_0\in\mathbb{R}^{n+1},r>0$ we have from \cite{Brakke1978}, 
    \begin{equation}
        (\partial_t - \Delta)|F(x,t)-x_0|^2 = -2n.
    \end{equation}
    Let $0\leq t_0\leq t_1<T$ with $t_1-t_0=r^2$, choose a time cutoff function $\gamma\in C^{\infty}([0,T))$ such that \begin{align*}
        \gamma (t) = \begin{cases}
            1, \quad \quad \quad t_1 \leq t < T\\
            \in[0,1], \: \: \: t_0 \leq t \leq t_1\\
            0, \quad \quad \quad t\leq t_0
        \end{cases}
    \end{align*} and $|\gamma'|\leq c_n/r^2$. Now define a space-time cutoff function like \cite{Ecker1995} $\eta(x,t): M \times [0,T) \rightarrow \mathbb{R}$ with  $\eta(x,t):= \psi(|F(x,t) - x_0|^2)\cdot \gamma(t)$ where $\eta(x,0)=0\forall x\in M, \psi$ is as in Definition 2.7. Furthermore as stated in \cite{Ecker1995},
    \begin{equation}
        |(\partial_t - \Delta)\eta| \leq \frac{\tilde{C}(n,\delta)}{r^2}.
    \end{equation}
\end{lemma}
\begin{proof}
    We write down evolution equation of $\eta$ according to Lemma 3.14 from \cite{Ecker2004} and then use (2.12):
\begin{align*}
    \frac{\partial \eta}{\partial t} &= \gamma'\varphi^{\frac{1}{\delta}} + \frac{1}{\delta}\gamma\varphi^{\frac{1}{\delta}-1}\varphi'\frac{\partial}{\partial t}|F(x,t)-x_0|^2\\
    \nabla \eta &= \frac{1}{\delta}\gamma\varphi^{\frac{1}{\delta}-1}\varphi'\nabla|F(x,t)-x_0|^2\\
    \Delta \eta &=\frac{1}{\delta}\gamma\varphi^{\frac{1}{\delta}-1}\varphi'\Delta|F(x,t)-x_0|^2 +\frac{1}{\delta}\gamma\varphi^{\frac{1}{\delta}-1}\varphi''|\nabla|F(x,t)-x_0|^2|^2 \\
    &+\frac{(1-\delta)}{\delta^2}\gamma\varphi^{\frac{1}{\delta}-2}(\varphi')^2|\nabla|F(x,t)-x_0|^2|^2\\
    |(\partial_t - \Delta)\eta|&\leq |\gamma'\varphi^{\frac{1}{\delta}}| + \frac{1}{\delta}|\gamma\varphi^{\frac{1}{\delta}-1}|\varphi'||(\partial_t-\Delta)|F(x,t)-x_0|^2| + \frac{1}{\delta}|\gamma\varphi^{\frac{1}{\delta}-1}|\varphi''||\nabla|F(x,t)-x_0|^2|^2\\
    &+\frac{|\delta-1|}{\delta^2}|\gamma\varphi^{\frac{1}{\delta}-2}|(\varphi')^2||\nabla|F(x,t)-x_0|^2|^2\\
&\leq \frac{c_n}{r^2} + \frac{2nc_n}{\delta r^2} + \frac{|(1-\delta)|c_n}{\delta^2 r^2} + \frac{c_n}{\delta r^2}
\end{align*}
and thus we obtain (2.13).
\end{proof}
We recall the Michael-Simon-Sobolev inequality which is crucial for Proposition 5.1.
\begin{theorem}
    (Theorem 2.1, Michael and Simon,\cite{Michael1973-tj}) Let $F:M^n \rightarrow\mathbb{R}^{n+1}$ be a complete, smooth proper immersion. For all $f\geq 0, f \in C_C^1(M)$, there exists a uniform constant $c_n$ such that 
    \begin{align}
        \left(\int_M f^{\frac{n}{n-1}} d\mu\right)^{\frac{n-1}{n}} \leq c_n \left(\int_M |\nabla f| + |H| f d\mu\right)
    \end{align}
    where $d\mu$ is the pulled back measure of $F(M)$ which uses the induced metric $g=F^*(\delta_{\text{Eucl}})$.
\end{theorem}
\section{Compact solutions}
\begin{theorem}
    Let $F:M^n \times [0,T)\rightarrow\mathbb{R}^{n+1}$ be a closed, smooth immersed mean curvature flow. Assume the following uniform bounds \[ \sup_{M^n \times [0,T)} |H|(x,t) \leq \frac{1}{\sqrt{2}} , \quad \quad \quad \sup_{M^n \times [0,T)} |\nabla H|(x,t) \leq 1.\] Let $p\in[4,\infty)$, if $\lVert A(0) \rVert_{p,M_0} =: A_0<\infty$, then there exists constants $k=k(n,p)$ and $C_1=C_1(p)$ such that for every $t\in(0,T)$
    \[ \int_M |A|^p(t)d\mu_t \leq 2e^{kt}\biggl(\int_M |A|^p(0)d\mu_0 + C_1 V\biggr)\]
    where $V:= \text{vol}_{g_0}(M_0) = \int_M d\mu_0 < \infty$.
\end{theorem}
\begin{proof}
Throughout we refer to a time-dependent $|A|(\cdot,t), |H|(\cdot, t)$ but for ease of writing we choose to not explicitly write the variable. From (2.4) and the uniform bound $|\nabla H|\leq 1$ we have 
\begin{align}
    \partial_t\int_M |A|^p d\mu_t &= \int_M \partial_t |A|^p d\mu_t + \int_M|A|^p \partial_td\mu_t\notag =p\int_M|A|^{p-2}(\nabla^2H\cdot A + H\cdot\text{tr}(A^3))d\mu_t\notag \\
    &\leq p^2\int_M|A|^{p-2}|\nabla A||\nabla H|d\mu_t + pc_n\int_M|A|^{p+1}|H| d\mu_t  \notag \\
    &\leq\frac{p^4}{4}\int_M|A|^{p-4}|\nabla A|^2d\mu_t + \epsilon\int_M|A|^{p+2}H^2 d\mu_t + \left(1+\frac{p^2c^2_n}{4\epsilon}\right)\int_M|A|^p d\mu_t
\end{align}
using (2.2) and the uniform bounds $H^4 \leq \frac{1}{4},|\nabla H|\leq 1$,
\begin{align*}
    2\int_M|A|^{p+2}H^2 d\mu_t&= \int_M|A|^{p}(2|A|^2H^2)d\mu_t = \int_M|A|^{p}((\partial_t - \Delta)H^2 + 2|\nabla H|^2)) d\mu_t\\
    &= \partial_t\int_M|A|^{p}H^2 d\mu_t - \int_M\partial_t(|A|^p)H ^2 d\mu_t+ \int_M|A|^pH^4 d\mu_t - \int_M|A|^p\Delta H^2 d\mu_t\\
    &\quad \quad \quad \quad \quad \quad \quad \quad \quad \quad \quad \quad \quad \quad \quad \quad \quad \: \: + 2\int_M|A|^p|\nabla H|^2 d\mu_t\\
    &\leq \partial_t\int_M|A|^pH^2 d\mu_t - p\int_M|A|^{p-2}(\nabla^2H\cdot A + H\cdot\text{tr}(A^3))H^2 d\mu_t -\int_M|A|^p\Delta H^2 d\mu_t\\
    &\quad \quad \quad \quad \quad \quad \quad \quad \quad \quad \quad \quad \quad \quad \quad \quad \quad \: \: + \frac{9}{4}\int_M|A|^p d\mu_t \\
    &\leq \partial_t\int_M|A|^pH^2 d\mu_t + p^2\int_M|A|^{p-2}|\nabla A||\nabla H|H^2 d\mu_t + 2p\int_M|A|^{p-1}|\nabla H|^2|H| d\mu_t \\
&+ pc_n\int_M|A|^{p+1}|H|^3 d\mu_t + 2p\int_M|A|^{p-1}|\nabla A||\nabla H||H| d\mu_t + \frac{9}{4}\int_M|A|^p d\mu_t\\
&\leq \partial_t\int_M|A|^pH^2 d\mu_t + \frac{p^2}{2}\int_M|A|^{p-2}|\nabla A| d\mu_t + \left(2p\sqrt{n}+ \frac{9}{4}\right)\int_M|A|^p d\mu_t \\ & + pc_n\int_M|A|^{p+1}|H|^3 d\mu_t  + 2p\int_M|A|^{p-1}|\nabla A||\nabla H||H| d\mu_t \\
&\leq \partial_t\int_M|A|^pH^2 d\mu_t + \frac{p^4}{4}\int_M|A|^{p-4}|\nabla A|^2 d\mu_t + \left(2pn + \frac{5}{2} + \frac{p^2c^2_n}{4}\right)\int_M|A|^p d\mu_t \\
&+ \frac{1}{4}\int_M|A|^{p+2} H^2 d\mu_t + 2p\int_M|A|^{p-1}|\nabla A||\nabla H||H| d\mu_t 
\end{align*}
which means,
\begin{align*}
\frac{7}{4}\int_M|A|^{p+2}H^2 d\mu_t &\leq \partial_t\int_M |A|^pH^2d\mu_t + \frac{p^4}{4}\int_M|A|^{p-4}|\nabla A|^2 d\mu_t + 2p\int_M|A|^{p-1}|\nabla A||\nabla H||H| d\mu_t\\
&+ \left(2pn + \frac{p^2c^2_n}{4} + \frac{5}{2}\right)\int_M|A|^p d\mu_t
\end{align*}
for $|A|\neq 0$\: using Young's inequality,
\[ 2p|A|^{p-1}\cdot|\nabla A|\cdot|\nabla H||H|\leq 4p^2|A|^{p-4}|\nabla A|^2 + \frac{1}{4}|A|^{p+2}H^2|\nabla H|^2\]
and since $|\nabla H|\leq 1$,
\begin{align}
    \frac{3}{2}\int_M|A|^{p+2}H^2 d\mu_t \leq \partial_t\int_M|A|^pH^2 d\mu_t + \left(\frac{p^4}{4} + 4p^2\right)\int_M|A|^{p-4}|\nabla A|^2 d\mu_t + \left(2pn + p^2c^2_n + \frac{5}{2}\right)\int_M|A|^p d\mu_t
\end{align}
using (3.1) with $\epsilon=\frac{3}{2}$ and (3.2) we see
\begin{align}
    \partial_t\int_M|A|^p d\mu_t
    &\leq 2p^4\int_M|A|^{p-4}|\nabla A|^2 d\mu_t + \partial_t\int_M|A|^pH^2 d\mu_t  + (2pn + 4 + 2p^2c^2_n)\int_M|A|^p d\mu_t
\end{align}
to estimate the integral containing the $|\nabla A|^2$ term use (2.3),
\begin{align}
    2\int_M|A|^{p-4}|\nabla A|^2 d\mu_t &= \int_M|A|^{p-4}\Delta |A|^2 d\mu_t - \int_M|A|^{p-4}\partial_t(|A|^2) d\mu_t + 2\int_M|A|^p d\mu_t \notag \\
&= -2(p-4)\int_M|A|^{p-4}|\nabla A|^2 d\mu_t - \frac{2}{(p-2)}\partial_t\int_M|A|^{p-2} d\mu_t + 2\int_M|A|^p d\mu_t \notag \\
&  \quad \quad \quad \quad \quad \quad \quad \quad \quad \: \: -\frac{2}{(p-2)} \int_M |A|^{p-2}H^2 d\mu_t \notag \\
\Rightarrow \int_M|A|^{p-4}|\nabla A|^2 d\mu_t  &\leq - \frac{1}{(p-2)(p-3)}\partial_t\int_M |A|^{p-2} d\mu_t + \frac{1}{(p-3)}\int_M|A|^p d\mu_t 
\end{align}
using inequality (3.4) in (3.3) we obtain
\begin{align*}
    \partial_t\int_M|A|^p d\mu_t &\leq  -\frac{2p^4}{(p-2)(p-3)}\partial_t\int_M|A|^{p-2} d\mu_t + \partial_t\int_M|A|^pH^2 d\mu_t \\& +\left(2pn + 4 + 2p^2c^2_n + \frac{2p^4}{(p-3)}\right)\int_M|A|^p d\mu_t \\
    & \leq  -C_1\partial_t\int_M|A|^{p-2} d\mu_t + \partial_t\int_M|A|^{p}H^2 d\mu_t + C_2\int_M |A|^p d\mu_t 
\end{align*}
where $C_1=C_1(p),C_2=C_2(n,p)$ are constants.
\begin{align*}
    \Rightarrow\partial_t \int_M \left(|A|^p + C_1|A|^{p-2} - |A|^{p}H^2 \right) d\mu_t \leq C_2\int_M |A|^p d\mu_t
\end{align*}
let $k=C_1 \cdot C_2$, using $H^2\leq\frac{1}{2}$, we calculate 
\begin{align}
\partial_t  \int_M (|A|^p + C_1|A|^{p-2} - |A|^pH^2) d\mu_t &\leq \frac{k}{2}\int_M |A|^p d\mu_t + kC_1\int_M|A|^{p-2} d\mu_t \notag \\
&  + k\int_M|A|^pH^2 d\mu_t - k\int_M|A|^pH^2 d\mu_t\notag \\
&\leq \frac{k}{2}\int_M |A|^p d\mu_t + kC_1\int |A|^{p-2} d\mu_t \notag \\
& + \frac{k}{2}\int_M|A|^p d\mu_t - k\int|A|^pH^2 d\mu_t \notag \\
&\leq k\int_M (|A|^p + C_1|A|^{p-2} - |A|^pH^2) d\mu_t.
\end{align}
Let $f:=\int_M(|A|^p + C_1|A|^{p-2} - |A|^{p}H^2) d\mu_t$, then
\begin{align*}
    \partial_t f &\leq kf\\
    \Rightarrow f(t) &\leq e^{\int_0^tk}f(0)\\
    \Rightarrow \int_M (|A|^p(t) +C_1|A|^{p-2}(t) - |A|^{p}(t)H^2(t))d\mu_t &\leq e^{\int_{0}^tk}\biggl(\int_M (|A|^p(0) + C_1|A|^{p-2}(0) -|A|^{p}(0)H^2(0)) d\mu_0\biggr)
\end{align*}
in particular, for any $t\in(0,T)$,
\begin{align*}
    \int_M|A|^p(t) d\mu_t\leq e&^{kt}\biggl(\int_M (|A|^p(0) + C_1|A|^{p-2}(0)) d\mu_0 \biggr)
\end{align*}
via Young's inequality:
\[ C_1\int_M|A|^{p-2}(0)\cdot 1d\mu_0 \leq \left(\int_M|A|^p(0)d\mu_0 + \frac{2}{p}(C_1)^{p/2}\text{vol}_{g_0}(M_0)\right)\] this yields
\begin{align}
    \int_M |A|^p(t)d\mu_t &\leq 2e^{kt}\biggl(\int_M |A|^p(0)d\mu_0 + C_1V \biggr)
\end{align}
where $k=k(n,p),C_1=C_1(p)$ and $V:=\text{vol}_{g_0}(M_0)<\infty$.
\end{proof}
\section{Noncompact solutions and Local Results}
\begin{theorem}
 Let $F:M^n \times [0,T)\rightarrow\mathbb{R}^{n+1}$ be a complete, smooth properly immersed mean curvature flow. Assume the following uniform bounds \[ \sup_{M^n \times [0,T)} |H|(x,t) \leq \frac{1}{\sqrt{2}}, \quad \quad \quad \sup_{M^n \times [0,T)} |\nabla H|(x,t) \leq 1.\] Let $p\in[4,\infty), x_0\in\mathbb{R}^{n+1}, 0<r\leq 1, K:= T\cdot \sup_{M^n\times[0,T)}|H|, \hat{\Omega}_0:= F_0^{-1}(B_{r+K}(x_0))$. Define \[ A_0:=\lVert A(0) \rVert_{p,\hat{\Omega}_0}, \quad \quad V_0:=\text{vol}_{g_0}(\hat{\Omega}_0)=\int_{F_0^{-1}(B_{r+K}(x_0))}d\mu_0.\]  If $A_0<\infty,V_0<\infty$ then there exists constants $\tilde{k}=\tilde{k}(n,p)$ and $\tilde{C}_1=\tilde{C}_1(p,r)$ such that for every $t\in(0,T)$
    \[\int_{M_t\cap B_{r/2}(x_0)} |A|^p(t) d\mu_t \equiv \int_{F_t^{-1}(B_{r/2}(x_0))}|A|^p(t) d\mu_t \leq 2e^{\tilde{k}t}\biggl(\int_{F_0^{-1}(B_{r+K}(x_0))} |A|^p(0) d\mu_0+ \tilde{C}_1V_0\biggr).\]
\end{theorem}
\begin{proof}
Let $\phi_t$ be the cutoff function as defined in Definition 2.7. Using (2.4) and $|\nabla H|\leq1$,
\begin{align}
    \partial_t\int_{M} |A|^p\phi_t d\mu_t &= p\int_{M}|A|^{p-2}(\nabla^2H\cdot A + H\cdot\text{tr}(A^3))\phi_t d\mu_t + \int_{M}|A|^p\partial_t\phi_t d\mu_t \notag\\
    &\leq p^2\int_{M}|A|^{p-2}|\nabla A||\nabla H|\phi_t d\mu_t + p\int_{M}|A|^{p-1}|\nabla H||\nabla \phi_t|d\mu_t \notag \\
    &+ pc_n\int_{M}|A|^{p+1}|H|\phi_t d\mu_t + \int_{M}|A|^{p}|\partial_t\phi_t| d\mu_t  \notag \\
    &\leq \frac{p^4}{4}\int_{M}|A|^{p-4}|\nabla A|^2\phi_t d\mu_t + p\int_{M}|A|^{p-1}|\nabla \phi_t| d\mu_t + pc_n\int_{M}|A|^{p+1}|H|\phi_t d\mu_t\notag \\
    &+ \int_{M}|A|^{p}|\partial_t\phi_t| d\mu_t + \int_M|A|^p\phi_t d\mu_t.
\end{align}
From (2.9) and (2.10) along with Young's inequality we get:\newpage \noindent
\begin{align}
    |A|^p|\partial_t\phi_t| &\leq \frac{\tilde{C}(n,\delta)}{r}|A|^p|H|\phi_t^{1-\delta} \notag \\
    &\leq \frac{(|A|^p)^{\frac{p+1}{p}}|H|\phi_t}{\frac{p+1}{p}} + \left(\frac{\tilde{C}}{r}\right)^{p+1} \frac{(\phi_t^{-\delta})^{p+1}|H|\phi_t}{p+1} \notag \\
&\leq p|A|^{p+1}|H|\phi_t + \tilde{C}^{p+1}\frac{|H|}{r^{p+1}(p+1)}\phi_t^{1-p\delta-\delta}\\
p|A|^{p-1}|\nabla \phi_t| &\leq \frac{\tilde{C}(n,\delta)}{r}p|A|^{p-1}\phi_t^{1-\delta} \notag \\
&\leq \frac{p(|A|^{p-1})^{\frac{p}{p-1}}\phi_t}{\frac{p}{p-1}} + \left(\frac{\tilde{C}}{ r}\right)^{p}\frac{p(\phi_t^{-\delta})^p\phi_t}{p} \notag \\
&\leq p|A|^p\phi_t + \tilde{C}^{p}\frac{1}{r^{p+1}}\phi_t^{1-p\delta},
\end{align}
choose $\delta=\frac{1}{1000p}$ or $\delta = \frac{1}{p}$ where convenient, because $\text{supp}\phi_t\subseteq \Omega_t$ from (2.8) we have  
\[ \int_M \phi_t^{1-p\delta} d\mu_t\leq \text{vol}_{g_t}(\Omega_t), \quad \quad \quad \int_M \phi_t^{1-p\delta-\delta} d\mu_t \leq \text{vol}_{g_t}(\Omega_t); \] using the uniform bound on $|H|$ we see that the coefficients of the cutoff function depend only on  $p,n$, set $a=a(p,n)$,
\begin{align}
    \partial_t\int_{M} |A|^p\phi_t d\mu_t &\leq \frac{p^4}{4}\int_{M}|A|^{p-4}|\nabla A|^2\phi_t d\mu_t + (p+1)\int_{M}|A|^p\phi_t d\mu_t + (p+pc_n)\int_{M}|A|^{p+1}H\phi_t d\mu_t \notag \\
&\quad \quad \quad \quad \quad \quad \quad \quad \quad \quad \quad \quad + \frac{a}{r^{p+1}} \text{vol}_{g_t}(\Omega_t)\notag\\
    &\leq \frac{p^4}{4}\int_{M}|A|^{p-4}|\nabla A|^2\phi_t d\mu_t +e\int_{M}|A|^{p+2}H^2\phi_t d\mu_t + \left(p+1 + \frac{p^2(1+c_n)^2}{4e}\right)\int_{M}|A|^p\phi_t d\mu_t\notag \\
    & \quad \quad \quad \quad \quad \quad \quad \quad \quad \quad \quad \quad +\frac{a}{r^{p+1}}\text{vol}_{g_t}(\Omega_t).
\end{align}
 Using (2.2), (2.4) and the uniform bounds on $|H|,|\nabla H|$,
\begin{align*}
 2\int_{M} |A|^{p+2}H^2\phi_t d\mu_t
    &= \partial_t\int_{M}|A|^pH^2\phi_t d\mu_t - \int_{M}\partial_t(|A|^p)H^2\phi_t d\mu_t - \int_{M} |A|^pH^2\partial_t\phi_t d\mu_t\\
    & + \int_{M}|A|^pH^4\phi_t d\mu_t - \int_{M}|A|^p\Delta H^2\phi_t d\mu_t + 2\int_{M}|A|^p|\nabla H|^2\phi_t d\mu_t \\
    &= \partial_t\int_{M} |A|^pH^2\phi_t d\mu_t - p\int_{M} |A|^{p-2}(\nabla^2H\cdot A + H\cdot\text{tr}(A^3))H^2\phi_t d\mu_t  \\
    & - \int_{M} |A|^pH^2\partial_t\phi_t d\mu_t + \int_{M} |A|^pH^4\phi_t d\mu_t - \int_{M}|A|^p\Delta H^2\phi_t d\mu_t + 2\int_{M}|A|^p|\nabla H|^2\phi_t d\mu_t\\
    &\leq \partial_t\int_{M}|A|^pH^2\phi_t d\mu_t + p^2\int_{M}|A|^{p-2}|\nabla A||\nabla H|H^2\phi_t d\mu_t \\
     &+  2p\int_{M}|A|^{p-1}|H||\nabla H|^2\phi_t d\mu_t + p\int_{M}|A|^{p-1}|\nabla H||\nabla \phi_t|H^2 d\mu_t\\
     &+ pc_n\int_{M}|A|^{p+1}|H|^3\phi_t d\mu_t + \frac{1}{2}\int_{M} |A|^p|\partial_t\phi_t| d\mu_t  + \frac{9}{4}\int_{M} |A|^p\phi_t d\mu_t\\
     &+ 2p\int_{M}|A|^{p-1}|\nabla A||\nabla H||H|\phi_t d\mu_t + 2\int_{M}|A|^p|H||\nabla H||\nabla \phi_t| d\mu_t \\
     &\leq \partial_t\int_{M}|A|^pH^2\phi_t d\mu_t + \frac{p^2}{2}\int_{M}|A|^{p-2}|\nabla A|\phi_t d\mu_t + \left(2p\sqrt{n}+\frac{9}{4}\right)\int_{M}|A|^p\phi_t d\mu_t\\
     &+ (p\sqrt{n}+2)\int_{M} |A|^p|H||\nabla H||\nabla \phi_t| d\mu_t  +pc_n\int_{M}|A|^{p+1}|H|^3\phi_t  d\mu_t \\
     &+\frac{1}{2}\int_{M}|A|^p|\partial_t\phi_t| d\mu_t + 2p\int_{M}|A|^{p-1}|\nabla A||\nabla H||H|\phi_t d\mu_t
\end{align*}
\begin{align*}
     \Rightarrow\frac{7}{4}\int_{M}|A|^{p+2}H^2\phi_t d\mu_t &\leq \partial_t\int_{M}|A|^pH^2\phi_t d\mu_t + \frac{p^4}{4}\int_{M}|A|^{p-4}|\nabla A|^2\phi_t d\mu_t + \left(2pn + \frac{5}{2} + \frac{p^2c^2_n}{4}\right)\int_{M} |A|^p\phi_t d\mu_t\\
     &+ (pn+2)\int_{M} |A|^p|H||\nabla H||\nabla \phi_t| d\mu_t +\frac{1}{2}\int_{M}|A|^p|\partial_t\phi_t| d\mu_t \\
     &+ 2p\int_{M}|A|^{p-1}|\nabla A||\nabla H||H|\phi_t d\mu_t.
\end{align*}
The last term can be estimated directly using Young’s inequality:
\begin{align*}
    2p|A|^{p-1}|\nabla A||\nabla H||H|\phi_t &\leq \frac{1}{4}|A|^{p+2}H^2|\nabla H|^2\phi_t + 4p^2|A|^{p-4}|\nabla A|^2\phi_t
\end{align*}
using (4.2) and Young's inequality:
\begin{align*}
    \frac{1}{2}|A|^p|\partial_t\phi_t| &\leq \frac{1}{4}|A|^{p+2}H^2\phi_t + \frac{p^2}{4}|A|^p\phi_t + \frac{\tilde{C}^{p+1}}{r^{p+1}(p+1)}|H|\phi_t^{1-p\delta-\delta},
\end{align*}
using (4.3) and Young's inequality: 
\begin{align*}
    (pn+2)|A|^p|H||\nabla H||\nabla \phi_t| 
    &\leq (pn+2)|A|^{p+1}|H||\nabla H|\phi_t + (pn+2)\tilde{C}^{p+1}\frac{|H||\nabla H|}{r^{p+1}(p+1)}\phi_t^{1-p\delta-\delta}\\
    &\leq \frac{1}{4}|A|^{p+2}H^2|\nabla H|^2\phi_t + (pn+2)^2|A|^p\phi_t + (pn+2)\tilde{C}^{p+1}\frac{|H||\nabla H|}{r^{p+1}(p+1)}\phi_t^{1-p\delta-\delta}
\end{align*}
choosing $\delta=\frac{1}{1000p}$, applying the uniform bounds on $|H|,|\nabla H|$ along with (2.8) we get
\begin{align}
     \int_{M}|A|^{p+2}H^2\phi_t d\mu_t &\leq \partial_t\int_{M}|A|^pH^2\phi_t d\mu_t+ \left(\frac{p^4}{4}+4p^2\right)\int_{M}|A|^{p-4}|\nabla A|^2\phi_t d\mu_t \notag \\
 &+\left(2pn + \frac{5}{2} + \frac{p^2c^2_n}{4} + 2(pn+2)^2p^2\right)\int_{M} |A|^p\phi_t d\mu_t + \frac{a}{r^{p+1}}\text{vol}_{g_t}(\Omega_t).
\end{align}
Using (4.4) with $e=1$ and (4.5) we have
\begin{align}
    \partial_t\int_{M}|A|\phi_t d\mu_t &\leq 2p^4\int_{M}|A|^{p-4}|\nabla A|^2\phi_t d\mu_t + \partial_t\int_{M}|A|^pH^2\phi_t d\mu_t\notag\\
     &+ ((2n+1)p + 5 + p^2(2(pn+2)^2 + (1+c_n)^2))\int_{M}|A|^p\phi_t d\mu_t + \frac{a}{r^{p+1}}\text{vol}_{g_t}(\Omega_t).
\end{align}
To deal with the first term we use (2.3),\newpage \noindent
\begin{align*}
    2\int_{M}|A|^{p-4}|\nabla A|^2\phi_t d\mu_t
    &= \int_{M} |A|^{p-4}\Delta |A|^2\phi_t d\mu_t - \int_{M}|A|^{p-4}\partial_t(|A|^2)\phi_t d\mu_t + 2\int_{M}|A|^p\phi_t d\mu_t \\
    &= -2(p-4)\int_{M}|A|^{p-4}|\nabla A|^2\phi_t d\mu_t - 2\int_{M}|A|^{p-3}|\nabla A||\nabla \phi_t| d\mu_t + 2\int_{M}|A|^p\phi_t d\mu_t\\
    &-\frac{2}{p-2}\partial_t\int_{M}|A|^{p-2}\phi_t d\mu_t + \frac{2}{p-2}\int_{M}|A|^{p-2}\partial_t\phi_t d\mu_t - \frac{2}{p-2}\int_{M}|A|^{p-2}H^2\phi_t d\mu_t\\
    \Rightarrow\int_{M}|A|^{p-4}|\nabla A|^2\phi_t d\mu_t &\leq \frac{1}{p-3}\int_{M} |A|^{p-3}|\nabla A||\nabla \phi_t|d\mu_t - \frac{1}{(p-2)(p-3)}\partial_t\int_{M}|A|^{p-2}\phi_t d\mu_t\\
    &+ \frac{1}{(p-2)(p-3)}\int_{M}|A|^{p-2}|\partial_t\phi_t|d\mu_t + \frac{1}{(p-3)}\int_{M}|A|^p\phi_t d\mu_t 
\end{align*}
furthermore we have,
\[ \frac{1}{p-3}|A|^{p-3}|\nabla A||\nabla \phi_t| \leq \frac{1}{4}|A|^{p-4}|\nabla A|^2\phi_t + \frac{1}{(p-3)^2}|A|^{p-2}|\nabla \phi_t|^2\phi_t^{-1}\]
which means,
\begin{align*}
     \int_{M}|A|^{p-4}|\nabla A|^2\phi_t d\mu_t &\leq \frac{4}{3(p-3)^2}\int_{M}|A|^{p-2}|\nabla\phi_t|^2\phi_t^{-1}d\mu_t - \frac{4}{3(p-2)(p-3)}\partial_t\int_{M}|A|^{p-2}\phi_t d\mu_t\\
    &+ \frac{4}{3(p-2)(p-3)}\int_{M}|A|^{p-2}|\partial_t\phi_t| d\mu_t + \frac{4}{3(p-3)}\int_{M} |A|^p\phi_t d\mu_t.
\end{align*}
Put this in (4.6),
\begin{align}
    \partial_t\int_{M}|A|^p\phi_t d\mu_t &\leq \frac{-8p^4}{3(p-2)(p-3)}\partial_t\int_{M} |A|^{p-2}\phi_t d\mu_t+ \frac{8p^4}{3(p-2)(p-3)}\int_{M} |A|^{p-2}|\partial_t\phi_t| d\mu_t\notag\\
    &+ \frac{8p^4}{3(p-3)^2}\int_{M}|A|^{p-2}|\nabla \phi_t|^2\phi_t^{-1} d\mu_t + \partial_t\int_{M}|A|^pH^2\phi_t d\mu_t + \frac{a}{r^{p+1}}\text{vol}_{g_t}(\Omega_t)\notag\\
    &+\left(\frac{8p^4}{3(p-3)} + (2n+1)p + 5 + p^2(2(pn+2)^2 + (1+c_n)^2) \right)\int_{M}|A|^p\phi_t d\mu_t.
\end{align}
Using (4.2), (4.3) and Youngs's inequality again:
\begin{align*}
\int_{M}|A|^{p-2}|\nabla \phi_t|^2\phi_t^{-1}d\mu_t &\leq \int_{M}\frac{\tilde{C}^2}{r^2}|A|^{p-2}\phi_t^{1-2\delta}d\mu_t\\
&\leq \int_{M} \frac{(|A|^{p-2})^{\frac{p}{p-2}}\phi_t}{\frac{p}{p-2}}d\mu_t + \left(\frac{\tilde{C}^2}{r^2}\right)^{\frac{p}{2}}\int_M(\phi_t^{-2\delta})^{\frac{p}{2}}\phi_t d\mu_t\\
&\leq p\int_{M}|A|^p\phi_t d\mu_t + \frac{a}{r^{p+1}}\int_{M} \phi_t^{1-p\delta} d\mu_t\\
\int_{M}|A|^{p-2}|\partial_t\phi_t| d\mu_t&\leq \int_{M}\frac{\tilde{C}}{r}|H||A|^{p-2}\phi_t^{1-\delta}d\mu_t\\
&\leq \int_{M}\frac{\sqrt{n}\tilde{C}}{r}|A|^{p-1}\phi_t^{1-\delta}d\mu_t\\
&\leq p\int_{M}|A|^p\phi_t d\mu_t + \frac{a}{r^{p+1}}\int_{M}\phi_t^{1-p\delta}d\mu_t
\end{align*}
now choose $\delta=\frac{1}{p}$; using (2.8) we can rewrite (4.7) as 
\begin{align*}
    \partial_t\int_{M} (|A|^p\phi_t + C_1|A|^{p-2}\phi_t - |A|^pH^2\phi_t)d\mu_t &\leq (2C_1 + C_2 )\int_{M} |A|^p\phi_t d\mu_t+ \frac{(a(2C_1 + 1))}{r^{p+1}}\text{vol}_{g_t}(\Omega_t)  
\end{align*}
where $C_1=C_1(p),C_2=C_2(n,p),a=a(p,n)$. 
\begin{align}
    \Rightarrow\partial_t\int_{M} (|A|^p\phi_t  + C_1|A|^{p-2}\phi_t - |A|^pH^2\phi_t)d\mu_t &\leq k_1\int_{M}|A|^p\phi_t d\mu_t  + \frac{k_2}{r^{p+1}}\text{vol}_{g_t}(\Omega_t)
\end{align}
such that $k_1=k_1(p,n),k_2=k_2(p,n)$ are constants.
Let $\tilde{k}=k_1\cdot k_2$, using $H^2\leq\frac{1}{2}$ we have
\begin{align*}
    \partial_t\int_{M} (|A|^p\phi_t + C_1|A|^{p-2}\phi_t - |A|^{p}H^2\phi_t)d\mu_t &\leq  \frac{\Tilde{k}}{2}\int_{M}|A|^p\phi_t d\mu_t + \Tilde{k}C_1\int_{M} |A|^{p-2}\phi_t d\mu_t \\
    &+ \Tilde{k}\int_{M}|A|^pH^2\phi_t d\mu_t - \Tilde{k}\int_{M}|A|^pH^2\phi_t d\mu_t+ \frac{\Tilde{k}}{r^{p+1}}\text{vol}_{g_t}(\Omega_t) \\
    &\leq  \Tilde{k}\biggl(\int_{M}(|A|^p\phi_t + C_1\int_{M} |A|^{p-2}\phi_t - \int_{M}|A|^pH^2\phi_t)d\mu_t\\
    &+\frac{1}{r^{p+1}}\text{vol}_{g_t}(\Omega_t)\biggr).
\end{align*}
Let  $f:=\int_{M} (|A|^p + C_1|A|^{p-2} - |A|^{p}H^2)d\mu_t$, this means
\begin{align}
    \partial_t f &\leq \Tilde{k}(f+r^{-(p+1)}\text{vol}_{g_t}(\Omega_t) \notag \\ 
    \partial_t f &\leq \Tilde{k}(f+r^{-(p+1)}\text{vol}_{g_0}(\hat{\Omega}_0))
\end{align}
the above inequality follows from Lemma 2.8. Integrate (4.9) with respect to time. For any $t\in(0,T)$ and a choice of localisation function $\phi_t$ as defined in Definition 2.7, we get
\begin{align*}
    \int_{F_t^{-1}(B_{r/2}(x_0))}|A|^p(t) d\mu_t \leq e^{\int_{0}^t\Tilde{k}}\biggl(&\int_{F_0^{-1}(B_{r+K}(x_0))}(|A|^p(0) + C_1|A|^{p-2}(0)) d\mu_0\\
    &+ r^{-(p+1)}\text{vol}_{g_0}(\hat{\Omega}_0)\biggr) \\
    \leq e^{\Tilde{k}t}\biggl(&\int_{F_0^{-1}(B_{r+K}(x_0))}(|A|^p(0) + C_1|A|^{p-2}(0)) d\mu_0\\
    &+ r^{-(p+1)}\text{vol}_{g_0}(\hat{\Omega}_0)\biggr)
\end{align*}
this is true because $\text{supp}\phi_0\subseteq \hat{\Omega}_0:= F_0^{-1}(B_{r+K}(x_0))$. After a final application of Young's inequality on the $|A|^{p-2}(0)$ term,
\begin{equation}
     \int_{F_t^{-1}(B_{r/2}(x_0))}|A|^p(t) d\mu_t \leq 2e^{\tilde{k}t}\biggl(\int_{F_0^{-1}(B_{r+K}(x_0))} |A|^p(0) d\mu_0+ \tilde{C}_1V_0\biggr)
\end{equation}
where $\Tilde{k}=\Tilde{k}(n,p), \tilde{C}_1=\tilde{C}_1(p,r), V_0=\text{vol}_{g_0}(\hat{\Omega}_0),K:=T\cdot\sup_{M\times[0,T)}|H|$.
\end{proof}
\section{$L^\infty$ bound and Smooth Extension}
In this section, we give a Moser iteration framework similar to \cite{Le_Sesum_2009} so that it can be used on the $L^p$ estimates which we have computed in the previous sections. Whereas they assume an optimal integral bound on the mean curvature, we are in the setting of a uniformly bounded mean curvature. Moreover, while they work with embeddings, we adapt their results to proper immersions. Following the construction in Example 3 of \cite{Michael1973-tj} - which is consistent with our definition of norms in Definition 2.3, Definition 2.4 - we consider the space of compactly supported Lipschitz functions $\text{Lip}_C(M) := \{ h \in W^{1,\infty}(M) : \text{supp}\,h \text{ is compact} \}$ as such functions are dense in $W^{1,p}_C(M)$. Since $F$ is continuous, $F(\text{supp}h)$ is compact, for any $h \in \text{Lip}_C(M)$, there exists $x_0 \in \mathbb{R}^{n+1}$ and $r > 0$ such that $F(\text{supp}h) \subset \overline{B}_r(x_0)$, which means $\text{supp}h\subset F^{-1}(\overline{B}_r(x_0))$. Because $F$ is proper and $\overline{B}_r(x_0)$ is compact in $\mathbb{R}^{n+1}$, the preimage $F^{-1}(\overline{B}_r(x_0))$ is compact in $M$. We now state and prove our results, proceeding exactly as in \cite{Le_Sesum_2009} but with necessary minor modifications.
\begin{proposition}
    (c.f. Proposition 3.1, Le, Sesum, \cite{Le_Sesum_2009}). Let $F_t:M^n\rightarrow\mathbb{R}^{n+1},t\in[0,T)$ be a complete, smooth, properly immersed mean curvature flow with $|H|(x,t)\leq c_1,\: \forall (x,t)\in M\times[0,T)$. For all non negative functions $v\in\text{Lip}_C(M\times [0,T))$, there exists a constant $\tilde{c}=\tilde{c}(n,c_1,T)$ such that
    \[ \int_0^T\int_M v^{\frac{2(n+2)}{n}} d\mu_t dt \leq c_n\left(\sup_{t\in[0,T]}\int_M v^2d\mu_t\right)^{\frac{2}{n}}\biggl(\int_0^T\int_M |\nabla v|^2 d\mu_t dt + \tilde{c}\sup_{t\in[0,T]}\int_M v^2 d\mu_t\biggr)\]
    where $d\mu_t$ is the measure induced by the pullback metric.
\end{proposition}
\begin{proof}
 Apply the Micheal Simon Sobolev inequality to the function $f=v^{\frac{2(n-1)}{n-2}}, n>2$, we get 
    \begin{align*}
        \biggl(\int_M v^{\frac{2n}{n-2}}d\mu_t\biggr)^{\frac{n-1}{n}} &\leq c_n\biggl(\int_M |\nabla v|v^{\frac{n}{n-2}}d\mu_t + \int_M|H|v^{\frac{2(n-1)}{(n-2)}} d\mu_t\biggr)
    \end{align*}
    we follow the steps exactly as \cite{Le_Sesum_2009}, the only thing different to them is that here we use Hölder only on the first term and use the uniform mean curvature bound in the second term  
    \begin{align*}
        \biggl(\int_M v^{\frac{2n}{n-2}}d\mu_t\biggr)^{\frac{n-2}{n}}&\leq c_n\biggl(\biggl(\int_M|\nabla v|^2d\mu_t\biggr)^{\frac{1}{2}}\biggl(\int_M v^{\frac{2n}{n-2}}d\mu_t\biggr)^{\frac{1}{2}} + c_1 \int_Mv^{\frac{2(n-1)}{(n-2)}} d\mu_t\biggr)^{\frac{n-2}{n-1}}\\
        &\leq c_n\biggl(\biggl(\int_M|\nabla v|^2d\mu_t\biggr)^{\frac{n-2}{2(n-1)}}\biggl(\int_M v^{\frac{2n}{n-2}}d\mu_t\biggr)^{\frac{n-2}{2(n-1)}} + c\bigg(\int_Mv^{\frac{2(n-1)}{(n-2)}} d\mu_t\biggr)^{\frac{n-2}{n-1}}\biggr).
    \end{align*}
    where $c=c(c_1,n)$. We use the product Young's inequality $ab \leq \epsilon a^p + \epsilon^{-q/p}b^q, a,b,\epsilon>0, p,q>1$ such that $\frac{1}{p}+\frac{1}{q}=1$
with $a=\biggl(\int_M v^{\frac{2n}{n-2}}d\mu_t\biggr)^\frac{n-2}{2(n-1)}, b=\biggl(\int_M|\nabla v|^2 d\mu_t\biggr)^{\frac{n-2}{2(n-1)}}$, $p=\frac{2(n-1)}{n}, q=\frac{2(n-1)}{n-2}$ and $\epsilon=\frac{1}{2c_n}$ to arrive at 
\begin{align}
    \biggl(\int_M v^{\frac{2n}{n-2}}d\mu_t\biggr)^{\frac{n-2}{n}} &\leq c_n\biggl(\int_M|\nabla v|^2d\mu_t + c\bigg(\int_Mv^{\frac{2(n-1)}{(n-2)}} d\mu_t\biggr)^{\frac{n-2}{n-1}}\biggr).
\end{align}
 We use inequality (7.10) from \cite{GilbargTrudinger2001} which is of the form \[ \lVert u \rVert_{L^r} \leq e\lVert u \rVert_{L^s} + e^{-\theta}\lVert u \rVert_{L^t}\] where $t<r<s, \theta = \frac{\left(\frac{1}{t}-\frac{1}{r}\right)}{\left(\frac{1}{r}-\frac{1}{s}\right)}, e>0$ with $t=2, r=\frac{2(n-1)}{n-2}, s=\frac{2n}{n-2}$ and so in our case the exponents are simpler than \cite{Le_Sesum_2009} i.e $2 < \frac{2(n-1)}{n-2} < \frac{2n}{n-2}$. Hence we have \begin{equation} \biggl(\int_M v^\frac{2(n-1)}{(n-2)}d\mu_t\biggr)^{\frac{n-2}{2(n-1)}} \leq e\biggl(\int_M v^{\frac{2n}{n-2}}d\mu_t\biggr)^{\frac{n-2}{2n}} + e^{-\theta}\biggl(\int_M v^2 d\mu_t\biggr)^{\frac{1}{2}}\end{equation} with
$\theta=\frac{1-\frac{n-2}{n-1}}{\frac{n-2}{n-1}-\frac{n-2}{n}}=\frac{n}{n-2}$.
Square both sides, use the Peter-Paul inequality with $\varepsilon=1$. Then for a choice of $e^2=\frac{1}{4c_nc}$, (5.2) becomes \begin{equation}
    \left(\int_M v^{\frac{2(n-1)}{n-2}} d\mu_t\right)^{\frac{n-2}{n-1}} \leq \frac{1}{2c_nc}\biggl(\int_M v^{\frac{2n}{n-2}}d\mu_t\biggr)^{\frac{n-2}{n}} + 2(4c_nc)^{\theta}\biggl(\int_M v^2 d\mu_t\biggr).
\end{equation} Finally use (5.3) in (5.1),
\begin{align}
    \biggl(\int_M v^{\frac{2n}{n-2}}d\mu_t\biggr)^{\frac{n-2}{n}} \leq c_n  \left(\int_M |\nabla v|^2 d\mu_t + c\int_Mv^2d\mu_t\right).
\end{align}
Now integrate with respect to time, use Hölder and (5.4) to get
\begin{align}
    \int_0^T\int_M v^{\frac{2(n+2)}{n}}d\mu_t dt &= \int_0^T\int_M v^2 v^{\frac{4}{n}}d\mu_tdt \leq \int_0^T\biggl(\int_M v^{\frac{2n}{n-2}}d\mu_t\biggr)^{\frac{n-2}{n}}\biggl(\int_M v^2d\mu_t\biggr)^{\frac{2}{n}}dt \notag \\ 
    &\leq c_n\left(\sup_{t\in[0,T]}\int_M v^2d\mu_t\right)^{\frac{2}{n}}\biggl(\int_0^T\int_M |\nabla v|^2 d\mu_t dt + \tilde{c}\sup_{t\in[0,T]}\int_M v^2 d\mu_t\biggr)
\end{align}
and $\tilde{c}=\tilde{c}(n,c_1,T)=Tc$.
\end{proof}
In the next lemma we show that a local $L^q$ integrability of the potential $f$ is sufficient to gain integrability for the sub-solution $v$ via the modified Michael-Simon-Sobolev inequality, thus leading to a reverse Hölder inequality. Such a local space time estimate for functions $v$ serves as the starting point for the subsequent Moser bootstrapping procedure.
\begin{lemma}
    (c.f. Lemma 4.1, Le, Sesum, \cite{Le_Sesum_2009}). Consider \begin{equation} (\partial_t - \Delta) v \leq fv, v\geq 0\end{equation} with $\lVert f\rVert_{q,F_{(\cdot)}^{-1}(B_r(x_0)),[t_0,T)}:=X<+\infty,q>\frac{n+2}{2}$, where $F_{(\cdot)}$ is a complete, smooth, properly immersed mean curvature flow $t\mapsto F_t,t\in[t_0,T)$ such that $\sup_{[t_0,T)}\sup_{F^{-1}_t(B_r(x_0))}|H|\leq c_1$ for $x_0\in\mathbb{R}^{n+1},r>0$. Let $\beta \geq 2$, choose the cutoff function $\eta$ as defined in Lemma 2.11, then
    \begin{align*}
   \biggl(\int_{t_0}^T\int_{F_t^{-1}(B_r(x_0))} (\eta^2 v^\beta)^\frac{n+2}{n} d\mu_t dt\biggr)^{\frac{n}{n+2}}  &\leq (2C\Lambda(\beta)(1+X))^{1+\nu}\biggl(\int_{t_0}^T\int_{F_t^{-1}(B_r(x_0))}  \eta^2v^\beta d\mu_tdt\\
    &+ \int_{t_0}^T\int_ {F_t^{-1}(B_r(x_0))} v^\beta \biggl( 2\eta|(\partial_t - \Delta)\eta| + |\nabla \eta|^2 \biggr) d\mu_t dt \biggr)
    \end{align*}
and $C=C(c_n,c_1,T), \Lambda(\beta)=100\beta \geq 1, \nu = \frac{n+2}{2q-(n+2)}$. 
\end{lemma}
\begin{proof}
    We follow essentially the same argument as in the proof of Lemma 4.1 from \cite{Le_Sesum_2009}. However, our assumptions differ slightly: instead of an optimal integral bound on the mean curvature we have that for any $x_0\in\mathbb{R}^{n+1},r>0$, $\sup_{[t_0,T)}\sup_{F^{-1}_t(B_r(x_0))}|H|\leq c_1$. Accordingly, only minor modifications to their proof must be made. The initial part of the proof carries over verbatim from \cite{Le_Sesum_2009} and we begin from the stage where the test function has been suitably controlled. Let $q>\frac{n+2}{2}, \beta \geq 2$ and set $ \Lambda(\beta)=100\beta\geq 1$. Then for any $s\in(0,T]$ and test function $\eta^2v^{\beta-1}$ where $\eta$ is the cutoff function constructed in Lemma 2.11, we obtain\newpage \noindent
\begin{align*}
    \int_0^s\int_M|\nabla(\eta v^{\beta/2})|^2 d\mu_t dt + \int_{M} v^\beta \eta^2 d\mu_s &\leq \Lambda(\beta) \biggl(\int_0^s\int_M v^\beta(2\eta|(\partial_t -\Delta)\eta|+ |\nabla \eta|^2)d\mu_t dt\\
    &+ \int_0^s\int_M |f|\eta^2v^\beta d\mu_t dt\biggr)=: Y
\end{align*}
Plug in this estimate into (5.5) which is the Micheal-Simon-Sobolev type inequality with the non-negative function as $\eta v^{\beta/2}$; note that Proposition 5.1 applies to such a function since $|H|$ is uniformly bounded on its support
\begin{align*}
    \int_0^T\int_M (\eta v^{\frac{\beta}{2}})^\frac{2(n+2)}{n}d\mu_tdt = \int_0^T\int_M (\eta^2 v^\beta)^{\frac{n+2}{n}}d\mu_t dt &\leq \biggl(\sup_{t\in[0,T]}\int_M \eta^2v^\beta d\mu_t\biggr)^{\frac{2}{n}}\biggl(c_n\int_0^T\int_M|\nabla(\eta v^{\beta/2})|^2d\mu_t dt\\
    &+ \tilde{c}\sup_{t\in[0,T]}\int_M\eta^2v^\beta d\mu_t\biggr)\\
    \int_0^T\int_M (\eta^2v^\beta)^{\frac{n+2}{n}}d\mu_t dt &\leq Y^{\frac{2}{n}}(c_nY + \tilde{c} Y).
\end{align*}
where $\tilde{c} = \tilde{c}(n, c_1, T)$ is the constant from Proposition 5.1 determined by the uniform spacetime bound on the mean curvature. The above inequality is equivalent to
\begin{align*}
    \lVert \eta^2v^\beta\rVert_{\frac{n+2}{n},M,[0,T)} &\leq CY
\end{align*}
with $C=C(c_n,c_1,T)$. This means 
\begin{align*}
    \lVert \eta^2v^\beta \rVert_{\frac{n+2}{n},M,[0,T)} \leq C\Lambda(\beta)\biggl(\int_0^T\int_M  v^\beta(2\eta((\partial_t -\Delta)\eta)+ |\nabla \eta|^2)d\mu_t dt + \int_0^T\int_M |f|\eta^2v^\beta d\mu_t dt\biggr)
\end{align*}
In contrast to \cite{Le_Sesum_2009}, we only have a local spacetime $L^q$ bound on $f$. Recalling that the cutoff function $\eta$ is supported on the spacetime cylinder $W := \{ (x,t) : |F(x,t)-x_0|^2 \leq r^2, t \geq t_0 \}$, we apply Hölder's inequality to the second term on the right-hand side with respect to the spacetime measure $d\omega := d\mu_t dt$,
\begin{align}
    \lVert \eta^2v^\beta \rVert_{\frac{n+2}{n},F_{(\cdot)}^{-1}(B_r(x_0)),[t_0,T)}&\leq C\Lambda(\beta)\biggl(\int_0^T\int_M  v^\beta(2\eta|(\partial_t -\Delta)\eta|+ |\nabla \eta|^2)d\mu_t dt \notag \\
    &+ \lVert f \rVert_{q,F_{(\cdot)}^{-1}(B_r(x_0)),[t_0,T)}\lVert \eta^2v^\beta\rVert_{\frac{q}{q-1},F_{(\cdot)}^{-1}(B_r(x_0)),[t_0,T)}\biggr)
\end{align}
Since, $q>\frac{n+2}{2}$, it holds that $1<\frac{q}{q-1}<\frac{n+2}{n}$, with the help of the interpolation inequality (7.10) from \cite{GilbargTrudinger2001} we get \begin{equation}
     \lVert \eta^2v^\beta\rVert_{\frac{q}{q-1},F_{(\cdot)}^{-1}(B_r(x_0)),[t_0,T)} \leq \varepsilon \lVert\eta^2v^\beta\rVert_{\frac{n+2}{n},F_{(\cdot)}^{-1}(B_r(x_0)),[t_0,T)} + \varepsilon^{-\nu}\lVert\eta^2v^\beta\rVert_{1,F_{(\cdot)}^{-1}(B_r(x_0)),[t_0,T)}
     \end{equation}
with $\nu = \frac{n+2}{2q-(n+2)}$. Choosing $\varepsilon = \frac{1}{2C\Lambda(\beta)(1+X)}$ in (5.8) and using this in (5.7) yields
\begin{align}
    \biggl(\int_{t_0}^T\int_{F_t^{-1}(B_r(x_0))} (\eta^2 v^\beta)^\frac{n+2}{n} d\mu_t dt\biggr)^{\frac{n}{n+2}}  &\leq (2C\Lambda(\beta)(1+X))^{1+\nu}\biggl( \int_{t_0}^T\int_{F_t^{-1}(B_r(x_0))}  \eta^2v^\beta d\mu_tdt \notag\\
    &+ \int_{t_0}^T\int_ {F_t^{-1}(B_r(x_0))} v^\beta \biggl( 2\eta|(\partial_t - \Delta)\eta| + |\nabla \eta|^2 \biggr) d\mu_t dt\biggr).
\end{align}
\end{proof}
\noindent
Our goal is to employ the Moser iteration technique to bound the $L^\infty$-norm of $v$ over a smaller spacetime domain by a lower-order $L^{\beta_0}$-norm ($\beta_0 \geq 2$) defined over a larger domain. This nested domain approach is standard in obtaining local estimates for sub-solutions to (5.6). We proceed with the following setup:
\begin{align*}
    r_k&:= \frac{1}{2}+\frac{1}{2^{k+1}}, \\
    \rho_k&:=r_{k-1}-r_k=\frac{1}{2^{k+1}},\\
    t_k&:= \frac{1}{2}\left(1-\frac{1}{4^k}\right),\\
    t_k&-t_{k-1}=6\rho_k^2, \\
    \eta_k(x,t)&:=\psi_k(|F_t(x)-x_0|^2) \times \gamma_k(t)\\
    \psi_k&:=\varphi_k^{1/\delta}, 0<\delta<1\\
    \varphi_k(s)&:=\begin{cases}
        1, \quad \quad \quad \:\: s\leq r_k^2\\
        \in[0,1],\quad r_k^2\leq s \leq r_{k-1}^2\\
        0, \quad \quad \quad \:\: s\geq r_{k-1}^2
    \end{cases}\\
    \gamma_k(t)&:=\begin{cases}
        1, \quad \quad \quad \: \: t_k\leq t\leq 1\\
        \in[0,1], \quad t_{k-1}\leq t\leq t_k\\
        0, \quad \quad \quad \:\: t\leq t_{k-1}
    \end{cases}
\end{align*}
such that it satisfies $|\varphi'_k|\leq \frac{c_n}{\rho_k},|\varphi_k''|\leq \frac{c_n}{\rho_k^2}, |\gamma'_k|\leq\frac{c_n}{\rho_k^2}$. Also 
\begin{equation}
        \sup_{t\in[0,T]}\sup_{x\in M} \biggl(2\eta_k|(\partial_t - \Delta)\eta_k| + |\nabla \eta_k|^2 + \eta_k^2\biggr) \leq \frac{\tilde{C}(n,\delta)}{\rho_k^2} = c_n4^{k+1}
\end{equation} holds due to (2.12) and (2.13) for a fixed $\delta\in(0,1)$. We localize the integral to the spatial region $F_t^{-1}(B_{r_k}(x_0))$ for each $t\in[t_{k},1]$, by noting that $\eta_k = 0$  whenever $x\notin F_t^{-1}(B_{r_{k-1}}(x_0)),$ for all $t \in [t_{k-1}, 1]$. Consequently, $\eta_k$ is supported within nested space time regions defined by $W_k :=\{ (x,t) : |F(x,t)-x_0|^2 \leq r_{k-1}^2 , t\geq t_{k-1} \}$ for each iteration $k$.
\begin{lemma}
    (c.f. Lemma 5.2, Le, Seseum, \cite{Le_Sesum_2009}). Let $\lambda=\frac{n+2}{n}$ and $T\geq 1$. Let $q>\frac{n+2}{2}$ and $\beta_0\geq2$. If $v$ is a sub-solution to (5.6) with $\lVert f\rVert_{q,F_{(\cdot)}^{-1}(B_{r_0}(x_0)),[t_0,T)}:=X<\infty$ where $F_{(\cdot)}$ is a complete, smooth, properly immersed mean curvature flow $t\mapsto F_t,t\in[t_0,T)$ such that $\sup_{[t_0,T)}\sup_{F^{-1}_t(B_{r_0}(x_0))}|H|\leq c_1$ for $x_0\in\mathbb{R}^{n+1},r_0\geq 1,t_0\leq 0$, then there exists a constant $\hat{E}=\hat{E}(n,q,T,c_1,X,\beta_0)=((C(n,c_1,T)(1+X)200\beta_0\lambda)^{1+\nu}c_n)^{\frac{n^2}{\beta_0}}$ satisfying
    \begin{equation*}
    \sup_{t \in [\frac{1}{2},1]}\sup_{F_t^{-1}(B_{1/2}(x_0))} |v| \leq \hat{E}\left(\int_0^1\int_{F_t^{-1}(B_1(x_0))} v^{\beta_0} d\mu_t dt\right)^{\frac{1}{\beta_0}}. 
    \end{equation*}
\end{lemma}
\begin{proof}
Let $\beta\geq 2$. Since $\sup_{[t_0,T)}\sup_{F^{-1}_t(B_{r_0}(x_0))}|H|\leq c_1$, $\eta_k \equiv 1$ on $W_k$ and $\text{supp}(\eta_k) \subset W_{k-1}$, it follows from (5.9) that,
    \begin{align*}
       \left(\int_{t_k}^1\int_{F_t^{-1}(B_{r_k}(x_0))}(v^\beta)^{\frac{n+2}{n}} d\mu_t dt\right)^{\frac{n}{n+2}} &\leq \left(\int_{t_{k-1}}^1\int_{F_{t_{k-1}}^{-1}(B_{r_{k-1}}(x_0))} (\eta_k^2v^\beta)^\frac{n+2}{n} d\mu_t dt\right)^{\frac{n}{n+2}}\\
   &\leq (2C\Lambda(\beta)(1+X))^{1+\nu}\int_{t_{k-1}}^1\int_{F_{t_{k-1}}^{-1}(B_{r_{k-1}}(x_0))} v^\beta\biggl(2\eta_k|(\partial_t - \Delta)\eta_k| \\
   &\quad \quad \quad \quad \quad \quad \quad \quad \quad \quad + |\nabla\eta_k|^2 + \eta_k^2\biggr)d\mu_t dt\\
   &\leq (2C\Lambda(\beta)(1+X))^{1+\nu}c_n4^{k+1}\int_{t_{k-1}}^1\int_{F_t^{-1}(B_{r_{k-1}}(x_0))} v^\beta d\mu_t dt\\
   &\leq E4^{k-1}\beta^{1+\nu}\int_{t_{k-1}}^1\int_{F_t^{-1}(B_{r_{k-1}}(x_0))} v^\beta d\mu_t dt
    \end{align*}
    where we used (5.10) in the second last inequality and $\nu=\frac{n+2}{2q-(n+2)}$ from (5.8). In the last inequality we have made use of $\Lambda(\beta)=100\beta$ and so $E=(C(1+X)200)^{1+\nu}c_n4^2$. Also set $\lambda=\frac{n+2}{n}$,
    \begin{align*}
        \left(\int_{t_k}^1\int_{F_t^{-1}(B_{r_k}(x_0))} v^{\lambda\beta}d\mu_t dt\right)^{\frac{1}{\lambda\beta}} \leq E^{\frac{1}{\beta}}4^{\frac{k-1}{\beta}}\beta^{\frac{1+\nu}{\beta}}\biggl(\int_{t_{k-1}}^1\int_{F_t^{-1}(B_{r_{k-1}}(x_0))} v^\beta d\mu_t dt\biggr)^{\frac{1}{\beta}}.
    \end{align*}
    We want to set up a exponent sequence so that we can do Moser iteration, for this we replace $\beta$ with $\lambda^{k-1}\beta_0$ and iterate the norms with exponent $\beta_k=\lambda^k\beta_0$,
    \begin{align*}
        \left(\int_{t_k}^1\int_{F_t^{-1}(B_{r_k}(x_0))} v^{\lambda^k\beta_0}d\mu_t dt\right)^{\frac{1}{\lambda^k\beta_0}} &\leq E^{\frac{1}{\lambda^{k-1}\beta_0}}4^{\frac{k-1}{\lambda^{k-1}\beta_0}}(\lambda^{k-1}\beta_0)^{\frac{1+\nu}{\lambda^{k-1}\beta_0}}\biggl(\int_{t_{k-1}}^1\int_{F_t^{-1}(B_{r_{k-1}}(x_0))} v^{\lambda^{k-1} \beta_0} d\mu_t dt\biggr)^{\frac{1}{\lambda^{k-1}\beta_0}}\\
        &\leq (E\beta_0^{1+\nu})^{\frac{1}{\lambda^{k-1}\beta_0}}(\lambda^{1+\nu}4)^{\frac{k-1}{\lambda^{k-1}\beta_0}}\biggl(\int_{t_{k-1}}^1\int_{F_t^{-1}(B_{r_{k-1}}(x_0))} v^{\lambda^{k-1} \beta_0} d\mu_t dt\biggr)^{\frac{1}{\lambda^{k-1}\beta_0}}\\
        \left(\int_{t_k}^1\int_{F_t^{-1}(B_{r_k}(x_0))} v^{\beta_k}d\mu_t dt\right)^{\frac{1}{\beta_k}} &\leq(E\beta_0^{1+\nu})^{\frac{1}{\beta_{k-1}}}(\lambda^{1+\nu}4)^{\frac{k-1}{\beta_{k-1}}}\biggl(\int_{t_{k-1}}^1\int_{F_t^{-1}(B_{r_{k-1}}(x_0))} v^{\beta_{k-1}} d\mu_t dt\biggr)^{\frac{1}{\beta_{k-1}}}\\
        &\leq (E\beta_0^{1+\nu})^{\frac{1}{\beta_0}\sum_{j=0}^{k-1}\frac{1}{\lambda^j}}(\lambda^{1+\nu}4)^{\frac{1}{\beta_0}\sum_{j=0}^{k-1}\frac{j}{\lambda^j}}\biggl(\int_{t_0}^1\int_{F_t^{-1}(B_{r_0}(x_0))} v^{\beta_0} d\mu_t dt\biggr)^{\frac{1}{\beta_0}}
    \end{align*}
    Let $k\rightarrow \infty$, we have $$ \sum_{j=0}^\infty \frac{1}{\lambda^j} = \frac{\lambda}{\lambda -1} = \frac{n+2}{2} \leq n^2, \sum_{j=0}^\infty \frac{j}{\lambda^j} = \frac{\lambda}{(\lambda -1)^2} = \frac{(n+2)n}{4} \leq n^2$$ and so
    \begin{align}
\sup_{t \in [\frac{1}{2},1]}\sup_{F_t^{-1}(B_{1/2}(x_0))} |v| \leq (E&\beta_0^{1+\nu})^{\frac{n^2}{\beta_0}}(\lambda^{1+\nu}4)^{\frac{n^2}{\beta_0}}\biggl(\int_0^1\int_{F_t^{-1}(B_1(x_0))} v^{\beta_0} d\mu_t dt\biggr)^{\frac{1}{\beta_0}} \notag \\
\sup_{t \in [\frac{1}{2},1]}\sup_{F_t^{-1}(B_{1/2}(x_0))} &|v| \leq \hat{E}\biggl(\int_0^1\int_{F_t^{-1}(B_1(x_0))} v^{\beta_0} d\mu_t dt \biggr)^{\frac{1}{\beta_0}}
    \end{align}
    where $\hat{E}=\hat{E}(n,q,T,c_1,X,\beta_0)=(E(\beta_0\lambda)^{1+\nu}4)^\frac{n^2}{\beta_0}=((C(n,c_1,T)(1+X)200\beta_0\lambda)^{1+\nu}c_n)^{\frac{n^2}{\beta_0}}$. 
\end{proof}
\begin{theorem}
    Let $F:M^n\times[0,T)\rightarrow\mathbb{R}^{n+1}$ be a complete, smooth, properly immersed mean curvature flow. Assume the following uniform bounds \[ \sup_{M^n\times[0,T)}|H|(x,t) \leq \frac{1}{\sqrt{2}}, \quad \quad \quad \quad \sup_{M^n\times[0,T)} |\nabla H|(x,t) \leq 1.\] Let $s>n+2, T\geq 1, K=T\cdot\sup_{M^n\times[0,T)}|H|, x_0\in\mathbb{R}^{n+1}, \hat{\Omega}_0:=F_0^{-1}(B_{2+K}(x_0))$. Define \[ A_0:=\lVert A(0) \rVert_{s,\hat{\Omega}_0}, \quad \quad V_0:=\text{vol}_{g_0}(\hat{\Omega}_0)=\int_{F_0^{-1}(B_{2+K}(x_0))}d\mu_0.\] If $A_0<\infty, V_0<\infty$, then there exists constants $\tilde{E}=\tilde{E}(n,s,T,V_0), \tilde{k}=\tilde{k}(n,s)$ and $\tilde{C}_1=\tilde{C}_1(s)$ such that \begin{equation}
        \sup_{t\in[\frac{1}{2},1]}\sup_{F_t^{-1}(B_{1/2}(x_0))} |A| \leq \tilde{E}\biggl\{1+\left( \int_{F_0^{-1}(B_{2+K}(x_0))}|A|^s(0) d\mu_0 + \tilde{C}_1V_0\right)^{\tilde{s}}\left(\int_0^1 e^{\tilde{k}t} dt\right)^{\tilde{s}}\biggr\} 
    \end{equation}
     where $\tilde{s}=\tilde{s}(s,n)$.
\end{theorem}
\begin{proof}
    We know that 
    \[ (\partial_t - \Delta)|A|^2 = -2|\nabla A|^2 + 2|A|^4 \leq 2|A|^2 \cdot |A|^2, \]
    this is of the form (5.6) with $f=2|A|^2$; using $|H|\leq\frac{1}{\sqrt{2}},|\nabla H|\leq 1$ and (4.10) we have $\lVert f\rVert_{q,F_{(\cdot)}^{-1}(B_r(x_0)),[t_0,T)}=\lVert |A(t)|^2 \rVert_{q,F_{(\cdot)}^{-1}(B_1(x_0)),[0,1)}<\infty$ for $q>\frac{n+2}{2} \iff \lVert A(t) \rVert_{s,F_{(\cdot)}^{-1}(B_1(x_0)),[0,1)} := X < \infty$ for $s>n+2$. Fix $\beta_0=2$, a direct application of (5.11) yields,
    \begin{align*}
        \sup_{t \in [\frac{1}{2},1]}\sup_{F_t^{-1}(B_{1/2}(x_0))} |A| &\leq \hat{E}\biggl(\int_0^1\int_{F_t^{-1}(B_1(x_0))} |A|^{4}(t) d\mu_t dt\biggr)^{\frac{1}{4}}\\
        &\leq (1+X)^{\frac{(1+\nu)n^2}{4}}\tilde{E}(n,s,T)\biggl(\int_0^1  \int_{F_t^{-1}(B_{1}(x_0))} |A|^4(t) d\mu_t dt\biggr)^{\frac{1}{4}}\\
        &\leq \biggl(1+\int_0^1\int_{F_t^{-1}(B_1(x_0))} |A|^s(t) d\mu_t dt\biggr)^\frac{(1+\nu)n^2}{4s} \tilde{E}(n,s,T)\biggl(\int_0^1  \int_{F_t^{-1}(B_{1}(x_0))} |A|^4(t) d\mu_t dt\biggr)^{\frac{1}{4}}
    \end{align*}
    where $\nu=\frac{n+2}{s-(n+2)}$. In order to express everything in terms of $L^s$ norms we use Hölder on the second integral above; also observe from Lemma 2.8 $\int_0^1 \text{vol}_{g_t}(F_t^{-1}(B_1(x_0)))dt\leq\int_0^1\text{vol}_{g_0}(F_0^{-1}(B_{2+K}(x_0)))dt$, then we obtain
    \begin{align*}
        \sup_{t \in [\frac{1}{2},1]}\sup_{F_t^{-1}(B_{1/2}(x_0))} |A|&\leq \biggl(1+\int_0^1\int_{F_t^{-1}(B_1(x_0))} |A|^s(t) d\mu_t dt\biggr)^\frac{(1+\nu)n^2}{4s} \tilde{E}(n,s,T)\biggl\{\bigg(\int_0^1\int_{F_t^{-1}(B_1(x_0))} d\mu_tdt\biggr)^{\frac{1}{4}-\frac{1}{s}}\\
        & \quad \quad \quad \quad \quad \quad \quad \quad \quad \quad \quad \quad \quad \quad \quad \quad \quad \quad \quad \cdot \biggl(\int_0^1 \int_{F_t^{-1}(B_{1}(x_0))} |A|^s(t) d\mu_t dt\biggr)^{\frac{1}{s}}\biggr\}\\
        &\leq \tilde{E}(n,s,T,V_0)\biggl(1+\int_0^1\int_{F_t^{-1}(B_1(x_0))}|A|^s(t)d\mu_tdt\biggr)^{\frac{(1+\nu)n^2 +4}{4s}}
    \end{align*}
we now use (4.10),
\begin{align*}
    \sup_{t \in [\frac{1}{2},1]}\sup_{F_t^{-1}(B_{1/2}(x_0))} |A|&\leq\tilde{E}(n,s,T,V_0)2^{\frac{(1+\nu)n^2+4}{4s}-1}\biggl\{1+\biggl(\int_0^1\int_{F_t^{-1}(B_1(x_0))}|A|^s(0)d\mu_tdt\biggr)^{\frac{(1+\nu)n^2+4}{4s}}\biggr\}\\
    &\leq \tilde{E}(n,s,T,V_0)\biggl\{\ 1+ \biggl(\int_{F_0^{-1}(B_{2+K}(x_0))}|A|^s(0)d\mu_0 + \tilde{C}_1V_0\biggr)^{\tilde{s}}\biggl(\int_0^1e^{\tilde{k}t}dt\biggr)^{\tilde{s}}\biggr\}
\end{align*}
where $\tilde{E}(n,s,T,V_0)$ is a constant. Defining $\tilde{s}:=\frac{(1+\nu)n^2+4}{4s}$ completes the proof of the local curvature bound.
\end{proof}
\begin{remark}
     Other than uniform space time bounds on $|H|,|\nabla H|$ we do not need any other global assumptions on the initial hypersurface, the hypersurface at later times. We can restrict our attention to the local domain in the initial manifold, defined precisely by: $\hat{\Omega}_0 := F_0^{-1}(B_{2+K}(x_0))$. The apriori bound on $|A|$ then requires finiteness of the $L^s$ norm $\lVert A(0) \rVert_{s,\hat{\Omega}_0} =: A_0 < \infty$ and finiteness of the local volume $\text{vol}_{g_0}(\hat{\Omega}_0) =: V_0 < \infty$. By virtue of equation (4.10), Lemma 2.8 the $L^s$-norm of $|A|$ at later times and the local volume at later times remain controlled by these quantities. Therefore, the validity of the estimate is purely a local property and is completely independent of the volume growth or curvature behavior outside of this domain for the entire duration of the flow.
\end{remark}
\begin{theorem}
    Let $F:M^n\times[0,T)\rightarrow\mathbb{R}^{n+1}$ be a complete, smooth, properly immersed mean curvature flow such that \[ \sup_{M^n\times[0,T)} |H|(x,t) \leq \frac{1}{\sqrt{2}}, \quad \quad \quad \quad \sup_{M^n\times[0,T)} |\nabla H|(x,t) \leq 1. \] Let $s>n+2, T> 1, \rho>2\sqrt{2}T$ be given. For all $x_0\in\mathbb{R}^{n+1}$ define $\hat{\Omega}_0(x_0):= F_0^{-1}(B_{\rho}(x_0))$, \[ A_0(x_0):=\lVert A(0) \rVert_{s,\hat{\Omega}_0(x_0)}, \quad \quad V_0(x_0):=\text{vol}_{g_0}(\hat{\Omega}_0(x_0))=\int_{F_0^{-1}(B_{\rho}(x_0))}d\mu_0.\] If $\mathcal{A}:=\sup_{x_0\in\mathbb{R}^{n+1}}A_0(x_0)<\infty$,$\mathcal{V}:=\sup_{x_0\in\mathbb{R}^{n+1}}V_0(x_0)<\infty$, then $\exists \kappa=\kappa(n,s,T,\rho,\mathcal{A},\mathcal{V})$ and a $T_0=T_0(n,\rho,\kappa)>0$ such that the flow admits a  complete, smooth, properly immersed extension on the interval $[T,T+T_0)$, which is unique among the class of such extensions with $\sup_{M^n\times[T/2,T+T_0)}|A|<\infty$.
\end{theorem}
\begin{proof}
 \textit{Step 1.} Let $\rho>2\sqrt{2}T$ be given. For any $x_0\in\mathbb{R}^{n+1}$, using $|H|\leq\frac{1}{\sqrt{2}},|\nabla H|\leq 1$, a direct application of Theorem 5.4 on $I_0=[0,1]$ along with $A_0(x_0), V_0(x_0)$ and $F_0^{-1}(B_{\rho}(x_0))$, gives (5.12). Moreover, we obtain $$\sup_{F_1^{-1}(B_{\rho/4}(x_0))}|A|(\cdot,1) \leq \Lambda(n,s,T,\mathcal{A},\mathcal{V}).$$ Any ball $B_{\rho}(x_0)$ can be covered by finitely many $B_{\rho/4}(y_i)$ with $y_i\in B_{\rho}(x_0)$, on each of which the same uniform pointwise estimate holds. Additionally, Lemma 2.8 via the uniform bound on the mean curvature yields for any $t<T$, $\int_{F_t^{-1}(B_{\rho}(x_0))}d\mu_t \leq \int_{F_0^{-1}(B _{\rho + T|H|}(x_0))}d\mu_0\leq V_{max}$ where $V_{max}\leq \tilde{N}\mathcal{V}, \tilde{N}=\tilde{N}(\rho,T,|H|)$. Combining all this we get $$\int_{F_1^{-1}(B_{\rho}(x_0))}|A|^s(\cdot,1)d\mu_1\leq\Lambda^s\hat{C}(n,s,\tilde{N},\mathcal{V}).$$ Applying Theorem 5.4 on the next sub-interval $I_1=[1/2,3/2]$ along with the above $L^s$ bound, we conclude $\sup_{t\in[1,3/2]}\sup_{F_t^{-1}(B_{\rho/4}(x_0))}|A|\leq \kappa(n,s,T,\rho,\mathcal{A},\mathcal{V})$. By defining $t_j:=\frac{j}{2}, j\in\mathbb{N}_0$, $I_j = [t_j, t_{j+2}] = \left[ \frac{j}{2}, \frac{j+2}{2} \right]$ and $N(T) = \max \{ j \in \mathbb{N}_0 : \frac{j+2}{2} < T \}$ such that for all $j\in\{0,...,N\}, \: I_j\subset [0, T)$, we can repeat the entire argument on each $I_j$ and $\int_{F_{t_{j+1}}^{-1}(B_{\rho}(x_0))}|A|^s(\cdot,t_{j+1})d\mu_{t_{j+1}}< \infty$ for finitely many steps. The constants $\kappa, \Lambda$ and  $\hat{C}$ are independent of $x_0,j$.\\
 \\
 \textit{Step 2.} It remains to treat $[\frac{N+2}{2}, T)$. For any $t$ in this interval: by definition $N\geq 0$, this means that $t\geq\frac{N+2}{2}\geq 1$ and we observe that the backward sub-interval $I_t^* := [t - 1, t]\subset[0, T)$. The \(L^s\)-bound over the local domain $F_{t-\frac{1}{2}}^{-1}(B_{\rho}(x_0))$ at time \(t-\frac12\) is supplied by the previous step because \(t-\frac12\geq \frac{N+1}{2}\). Thus, Theorem 5.4 applies and the estimates guarantee $$\sup_{t \in [1/2, T)} \sup_{F_t^{-1}(B_{\rho/4}(x_0))} |A| \leq \kappa< \infty.$$ \textit{Step 3.} We can use the local interior estimate from proposition 3.22 of \cite{Ecker2004} - this holds for complete smooth properly immersed due to an identical localization argument - to get uniform bounds on all higher covariant derivatives of $|A|$: $$\sup_{[1/2,T)}\sup_{F_t^{-1}(B_{\rho/8}(x_0))} |\nabla^m A| \leq C(m,n,\rho,\kappa),\forall m\geq 1.$$ Now fix $x\in M$, choose $t_0\in(T/2,T)$ and set $x_0:=F(x,t_0)$. Then, for every $t\in(T/2,T)$, the uniform bound on the mean curvature gives us $|F(x,t)-x_0|\leq|H||t-t_0|<\frac{T}{2\sqrt{2}}$. Since $\rho>2\sqrt{2}T$, it is true that $|F(x,t)-x_0|\leq\frac{\rho}{8}, x\in F_t^{-1}(B_{\rho/8}(x_0))$ and so all the above estimates hold. Furthermore, the metric is uniformly equivalent in time $e^{-\sqrt{2}\kappa|t_1-t_2|}g(t_1)\leq g(t_2) \leq e^{\sqrt{2}\kappa|t_1-t_2|}g(t_1),\:\forall t_1,t_2\in(T/2,T)$. \\
 \\
 \textit{Step 4.} Let $K\Subset M$ be compact. The preceding estimates lead to uniform $C^k(K)$ bounds for $F_t|_K$ for every $k\geq0$. We also get equicontinuity of $F_t|_K$ in time. Since $K$ was arbitrary, by Arzelà--Ascoli and the uniqueness of the $C^0(K)$ limit, $F_t|_K$ converges smoothly as $t\rightarrow T$ to a smooth limiting immersion $F_T:M\rightarrow\mathbb{R}^{n+1}$ having the same curvature bound \begin{equation}
     \sup_{F_T^{-1}(B_{\rho/8}(x_0))}|A|\leq\kappa.
 \end{equation}
 \textit{Step 5.} In particular, $F_T$ fulfills the uniform local Lipschitz condition required by Theorem 4.2 of \cite{Ecker1991} and so there exists a smooth immersed mean curvature flow \begin{align}
 \tilde{F}:M\times[0,T_0)&\rightarrow\mathbb{R}^{n+1}\notag \\
 \tilde{F}(\cdot,0)&=F_T
 \end{align} 
 where $T_0=T_0(n,\rho,\kappa)>0$ is the time of existence as produced by an exhaustion argument in the proof of Theorem 4.2 via Proposition 4.1 of \cite{Ecker1991}. Although their result states only Hölder continuity at $t=0$ for general Lipschitz initial data, since $F_T$ here is smooth with uniform bounds on $|A|,|\nabla^mA|$, this implies that $\tilde{F}$ is smooth up to and including the time $t=0$. We then define: \[ \hat{F}(x,t):= \begin{cases}
     F(x,t) \quad &t\in[0,T),\\
     \tilde{F}(x,t-T) \quad &t\in[T,T+T_0).
 \end{cases}\] 
\textit{Step 6.} Theorem 1.1 from \cite{ChenYin2007} asserts that $\hat{F}$ is the unique smooth extension of the flow $F$ on the interval $[T,T+T_0)$ with bounded second fundamental form on $[T/2, T+T_0)$: From (5.13), (5.14) and the fact that \(F_T\) is smooth, we obtain via Theorem 4.2 and Proposition 4.1 of \cite{Ecker1991}
\[
\sup_{M\times[0,T_0)}|\tilde A|\leq C(n,\rho,\kappa),
\]
and
\[
\sup_{M\times[T/2,T+T_0)}|\hat A|(\cdot,t)
\leq
\max\{\kappa,C(n,\rho,\kappa)\}<\infty.
\]
Therefore $\hat{F}$ belongs to the uniqueness class of \cite{ChenYin2007}; it is the unique, smooth extension of $F$ on $[T,T+T_0)$ with bounded $|A|$. See also \cite{han2026wellposedness} for a new proof of existence and uniqueness of immersed MCFs with bounded $|A|$.\\
\\
\textit{Step 7.} These solutions inherit the properness of the initial immersion due to the uniform mean curvature bound: $\forall t \in [0, T+T_0)$, for any $Q \subset \mathbb{R}^{n+1}$ compact, let $y\in \mathbb{R}^{n+1}$ and $R>0$ such that $Q \subseteq \overline{B}_R(y)$, then by Proposition 2.6 we have the containment$$ F_t^{-1}(Q) \subseteq F_t^{-1}(\overline{B}_R(y)) \subseteq F_0^{-1}(\overline{B}_{R+\frac{T+T_0}{\sqrt{2}}}(y)). $$ Since $F_0$ is a proper immersion, the set $F_0^{-1}(\overline{B}_{R+\frac{T+T_0}{\sqrt{2}}}(y))$ is a compact subset of $M$. By continuity of $F_t$ the preimage $F_t^{-1}(Q)$ is a closed subset of this compact set, and hence compact; it follows that a smooth proper immersion with bounded second fundamental form is complete.
\end{proof}

\bibliographystyle{alpha}
\bibliography{mcfextn}

\end{document}